\documentclass{article}
\usepackage{amsmath,amsthm}
\usepackage{amssymb}
\usepackage{mathtools}
\usepackage{float}
\usepackage{geometry}
\usepackage{graphicx}
\usepackage{hyperref,url,cite}
\usepackage{defs,fonts}

\newtheorem{theorem}{Theorem}[section]

\newtheorem{proposition}[theorem]{Proposition}

\theoremstyle{definition}

\newtheorem{remark}[theorem]{Remark}

\newcommand{\ep}{\varepsilon}

\newcommand{\norm}[1]{\left\lVert#1\right\rVert}
\numberwithin{equation}{section}

\begin{document}

\title{A Space-time Nonlocal Traffic Flow Model: Relaxation Representation and Local Limit\thanks{This rsearch is supported in part by US NSF
DMS-1937254, DMS-2012562, and CNS-2038984.
}}

\author{Qiang 
Du\thanks{Department of Applied Physics and Applied Mathematics and Data Science Institute, Columbia University, New York, NY 10027, USA, Email: 
\textup{\nocorr
      \texttt{qd2125@columbia.edu}}}
\and 
 Kuang Huang\thanks{Department of Applied Physics and Applied Mathematics, Columbia University, New York, NY 10027, USA, Email: \textup{\nocorr
      \texttt{kh2862@columbia.edu}}
      } \textsuperscript{,}\thanks{Corresponding author.}
      \and James Scott\thanks{Department of Applied Physics and Applied Mathematics, Columbia University, New York, NY 10027, USA, Email: \textup{\nocorr
      \texttt{jms2555@columbia.edu}}
      }
      \and
Wen Shen\thanks{Department of  Mathematics, Pennsylvania State University, University Park, PA 16802, USA, Email: \textup{\nocorr
      \texttt{wxs27@psu.edu}}
      }
}
\date{}
\maketitle

{\let\thefootnote\relax\footnote{{ \emph{2020 Mathematics Subject Classification.} 35L65, 90B20, 35R09. }}}
{\let\thefootnote\relax\footnote{{ \emph{Key words.} Traffic flow modelling, nonlocal conservation law, time-delay, hyperbolic system with relaxation, nonlocal-to-local limit. }}}

\begin{center}
{\small \color{purple} This work has been published in \emph{Discrete and Continuous Dynamical Systems}, 2023, 43(9): 3456-3484. Please refer to the official publication for citation.}
\end{center}

\begin{abstract}
We propose and study a nonlocal conservation law modelling traffic flow in the existence of inter-vehicle communication. It is assumed that the nonlocal information travels at a finite speed and the model involves a space-time nonlocal integral of weighted traffic density. The well-posedness of the model is established under suitable conditions on the model parameters and by a suitably-defined initial condition.
In a special case where the weight kernel in the nonlocal integral is an exponential function, the nonlocal model can be reformulated as a $2\times2$ hyperbolic system with relaxation.
With the help of this relaxation representation, we show that the Lighthill-Whitham-Richards model is recovered in the equilibrium approximation limit.
\end{abstract}


\section{Introduction}

\subsection{The nonlocal space-time traffic flow model}\label{sub:model}
We consider the following nonlocal conservation law modeling traffic flow
\begin{align}
  & \partial_t\rho(t,x) + \partial_x(\rho(t,x) v(q(t,x))) = 0, \quad  x\in \mathbb{R},\, t>0, \label{eq:model_nonlocal_1}\\
& \mbox{where}\qquad q(t,x) = \int_0^\infty \rho(t-\gamma s, x+s) w(s)\,ds. \label{eq:model_nonlocal_2}
\end{align}
Here, the quantity $\rho(t,x)\in[0,1]$ represents the traffic density, where $\rho=0$ indicates an empty road ahead and $\rho=1$ models bumper-to-bumper traffic jam.
The nonlocal quantity $q(t,x)$ is a weighted average of $\rho(t^\ast,x^\ast)$ along the space-time path
$$t^\ast=t-\gamma s, \quad x^\ast=x+s, \qquad \mbox{for}~s\in[0,\infty),$$
with an averaging kernel $w=w(s)$.
The vehicle velocity $v=v(q(t,x))$ depends on the nonlocal traffic density $q(t,x)$ through a decreasing function $v(\cdot)$.
The model \eqref{eq:model_nonlocal_1}-\eqref{eq:model_nonlocal_2} is the evolution associated to a past-time condition $\rho(t,x)$ given on the half plane $t\le 0$.

\subsection{Background and motivation}\label{sub:motivation}
The model \eqref{eq:model_nonlocal_1}-\eqref{eq:model_nonlocal_2} takes inspiration from the classical Lighthill-Whitham-Richards (LWR) model \cite{lighthill1955kinematic,richards1956shock}
\begin{align}\label{eq:local_model}
  \partial_t\rho(t,x) + \partial_x(\rho(t,x) v(\rho(t,x))) = 0,  \quad  x\in \mathbb{R},\, t>0,
\end{align}
in which the vehicle velocity $v=v(\rho(t,x))$ depends only on the local traffic density $\rho(t,x)$.
The LWR model \eqref{eq:local_model} is a scalar conservation law with the flux function $f(\rho) \doteq \rho v(\rho)$.
In the instance of inter-vehicle communication \cite{dey2016vehicle}, the flux may have a nonlocal dependence on traffic density in order to capture each vehicle's reaction to downstream traffic conditions. It is useful to incorporate time delays of this traffic density information in the distance \cite{yang2014control,keimer2019nonlocal}.
In \eqref{eq:model_nonlocal_1}-\eqref{eq:model_nonlocal_2}, we incorporate both nonlocal fluxes and time delays via velocities that depend on a weighted space-time average of the traffic density, assuming that the traffic density information travels at a constant speed $\gamma^{-1}$.

If the choice of rescaled weights $w_\ep(s) = \ep^{-1} w(s/\ep)$ is made, then formally the equations \eqref{eq:model_nonlocal_1}-\eqref{eq:model_nonlocal_2} converge to the local equation \eqref{eq:local_model} as $\varepsilon \to 0$.
The main goal of this paper is to demonstrate this in a rigorous manner via convergence of solutions.

There has recently been much research interest in nonlocal effects in phenomena described by conversation laws; there is a wide variety of applications but a dearth of analytical understanding.
Some application areas from which nonlocal conservation laws arise are traffic flows \cite{Lee2019,Lee2019a,chiarello2020micro,karafyllis2020analysis,jin2020generalized,Chiarello2018,chiarello2019junction,chiarello2019non}, sedimentation \cite{Betancourt2011}, pedestrian traffic \cite{Colombo2012,burger2020non}, material flow on conveyor belts \cite{Goettlich2014,rossi2020well}, and the numerical approximation of local conservation laws \cite{du2012new,du2017nonlocal,du2017numerical,fjordholm2021second}.

For several traffic flow models, the nonlocal mechanism is introduced in the flux term.
One such 
model that recovers the LWR model \eqref{eq:local_model} when the effect is localized
was proposed in \cite{Blandin2016,goatin2016well}:
\begin{align}
  & \partial_t\rho(t,x) + \partial_x(\rho(t,x) v(q(t,x))) = 0, \quad  x\in \mathbb{R},\, t>0, \label{eq:model_nonlocal_space_1}\\
  &\mbox{where} \qquad  q(t,x) = \int_0^\infty \rho(t, x+s) w(s)\,ds. \label{eq:model_nonlocal_space_2}
\end{align}
Various analytical aspects of this model have been investigated, including the existence and uniqueness of solutions \cite{Blandin2016,goatin2016well,bressan2019traffic}, existence and stability of traveling wave solutions \cite{ridder2019traveling,shen2018traveling}, development of numerical schemes \cite{goatin2016well,Chalons2018,friedrich2019maximum}, and stability analysis of the model in the case where the domain (road) is a closed ring \cite{huang2022stability}.
Convergence of solutions of \eqref{eq:model_nonlocal_space_1}-\eqref{eq:model_nonlocal_space_2} to  its local limit (which is the LWR model \eqref{eq:local_model}) was established in \cite{bressan2019traffic,bressan2020entropy} by way of an a priori BV estimate and an entropy estimate, both of which were obtained via reformulation of the nonlocal model as a $2\times 2$ relaxation system in the case of exponential weight kernels.
This is not the only mechanism that has been used to investigate the nonlocal-to-local limit; see the works of \cite{colombo2021local, keimer2019approximation, coclite2020general, colombo2022nonlocal,colombo2019singular,friedrich2022conservation}.

\subsection{Assumptions on the model}\label{sub:assumptions}
We conduct an analogous study of the nonlocal-to-local limit for the model
\eqref{eq:model_nonlocal_1}-\eqref{eq:model_nonlocal_2} with suitable choices of the functions $w,v$ and the past-time condition.
To fix ideas, we make the following assumptions on $w,v$:

\vspace{4pt}\noindent
\textbf{Assumption 1.} \textit{The velocity function $v\in \mathbf{C}^2([0,1])$ is strictly decreasing with $v(0) = v_{\mathrm{max}}$ and $v(1)=0$,
where $v_{\mathrm{max}}>0$ represents the maximum vehicle speed.}

\vspace{4pt}\noindent
\textbf{Assumption 2.} \textit{The weight kernel $w\in \mathbf{C}^1 ([0,\infty))$ is non-negative and
satisfies
\begin{align}\label{eq:kernel_decay}
\int_0^\infty w(s)\,ds=1 \quad \text{and} \quad w'(s)\leq -\beta w(s) \ \ \forall s\geq0
\end{align}
for a constant $\beta>0$.
}

The average density $q$ is taken along a space-time curve that requires traffic density data for all past times $t\leq0$. Therefore, the model
\eqref{eq:model_nonlocal_1}-\eqref{eq:model_nonlocal_2}
shall be equipped with a past-time condition on the lower half-plane, i.e.,
\begin{align}\label{eq:initial_negative}
  \rho(t,x)=\rho_{-}(t,x), \quad (t,x)\in(-\infty,0]\times\mathbb{R},
\end{align}
where $\rho_{-}\in\mathbf{L}^\infty((-\infty,0]\times\mathbb{R})$ is a given function.

\subsection{Main results}\label{sub:results}
Our first main result is the existence of Lipschitz solutions to the past-time value problem \eqref{eq:model_nonlocal_1}-\eqref{eq:model_nonlocal_2}-\eqref{eq:initial_negative} with Lipschitz past-time data $\rho_{-}$.
\begin{theorem}\label{thm:wellposedness_lipschitz}
  Suppose that Assumption 1 and Assumption 2 are satisfied and that
  \begin{equation}\label{eq:GammaSmallAssump}
  \gamma \leq \gamma_{\mathrm{max}} \doteq \min\left\{ \frac{1}{3(v_{\mathrm{max}}+\norm{v'}_\infty)}, \frac{\beta}{w(0) \norm{v'}_\infty} \right\}.
  \end{equation}
  Suppose that the past-time data $\rho_-$ is a bounded Lipschitz function belonging to the class $\mathcal{X}_{\mathrm{Lip},L}$; see definition \eqref{eq:Lip_IC_Def} below.
  Then the past-time value problem
  \eqref{eq:model_nonlocal_1}-\eqref{eq:model_nonlocal_2}-\eqref{eq:initial_negative} admits a solution $\rho$ that is Lipschitz continuous and satisfies \eqref{eq:model_nonlocal_1}-\eqref{eq:model_nonlocal_2}-\eqref{eq:initial_negative} pointwise.
  Furthermore, the solution $\rho$ satisfies the uniform bounds
  \begin{align}\label{eq:sol_low_upp_bounds}
    \rho_{\mathrm{min}} \leq \rho(t,x) \leq \rho_{\mathrm{max}} , \quad (t,x)\in[0,\infty)\times\mathbb{R},
  \end{align}
  where $\rho_{\mathrm{min}}$ and $\rho_{\mathrm{max}}$ are defined in \eqref{eq:sol_lower_bound}-\eqref{eq:sol_upper_bound} below and depend only on $\gamma$, $v$, $w$ and $\rho_-$.
\end{theorem}

Formally, as the time-delay parameter $\gamma$ approaches zero, the system \eqref{eq:model_nonlocal_1}-\eqref{eq:model_nonlocal_2} approaches the nonlocal-in-space system \eqref{eq:model_nonlocal_space_1}-\eqref{eq:model_nonlocal_space_2}. This is also true in a qualitative sense; each of the key estimates for \eqref{eq:model_nonlocal_1}-\eqref{eq:model_nonlocal_2} remain valid as $\gamma \to 0$, as the bounding constants neither vanish nor blow up.
Analogous statements of all of our results  hold for \eqref{eq:model_nonlocal_space_1}-\eqref{eq:model_nonlocal_space_2}, see \cite{bressan2019traffic}, and can be formally recovered from our results by taking $\gamma \to 0$. However, quantitatively stronger results hold for \eqref{eq:model_nonlocal_space_1}-\eqref{eq:model_nonlocal_space_2}. For example, the main estimates in Proposition \ref{prop:L1Stab:1} and Theorem \ref{thm:BV_bound} concern $\mathbf{L}^1$ estimates in space and time, whereas the analogous results for \eqref{eq:model_nonlocal_space_1}-\eqref{eq:model_nonlocal_space_2} hold for $\mathbf{L}^1$ in space and $\mathbf{C}^0$ in time; see again \cite{bressan2019traffic}.

The proof of Theorem \ref{thm:wellposedness_lipschitz} makes up Section~\ref{sec:existence_lip_sol}. We use a fixed point argument combined with the method of characteristics, which is heavily inspired by the proof of \cite{bressan2019traffic} for the existence of solutions to the nonlocal-in-space model.

In certain modelling applications, it might only be possible to gather the traffic data at a certain initial time.
In such a case, a natural choice of past-time data via the following extension of initial data:
\begin{align}
  \rho_{-}(t,x) &= \rho_0(x), \qquad (t,x) \in (-\infty,0] \times \mathbb{R},  \label{eq:extension_vertical}
\end{align}
for a given function $\rho_0 : \mathbb{R} \to \mathbb{R}$. We can then treat \eqref{eq:model_nonlocal_1}-\eqref{eq:model_nonlocal_2}-\eqref{eq:initial_negative} as an initial-value problem, since the quantity $q$
depends only on $t \in (0,\infty)$. To be precise, with \eqref{eq:extension_vertical}, the equation \eqref{eq:model_nonlocal_2} becomes
\begin{equation}\label{eq:model_nonlocal_3}
  q(t,x) = \int_0^{\infty} \rho_0 \Big( x+\frac{t}{\gamma}+s \Big) w \Big(s+\frac{t}{\gamma} \Big) ds + \int_0^{\frac{t}{\gamma}} \rho(t-\gamma s,x+s) w(s) ds.
\end{equation}
Under this consideration,
the equations \eqref{eq:model_nonlocal_1}-\eqref{eq:model_nonlocal_2}-\eqref{eq:initial_negative} where the past-time data is given by the equation \eqref{eq:extension_vertical} are equivalent to the Cauchy problem
\eqref{eq:model_nonlocal_1}-\eqref{eq:model_nonlocal_3} with the initial condition $\rho(0,x) = \rho_0(x)$.

For the particular choice \eqref{eq:extension_vertical} of past-time data, we establish the well-posedness of the Cauchy problem in the setting of weak solutions, which is our second main result.

\begin{theorem}\label{thm:WeakSolnExistence}
  Suppose that Assumption 1, Assumption 2 and \eqref{eq:GammaSmallAssump} are satisfied, and let $\rho_0(x)$ be a bounded function with finite total variation belonging to the class $\mathcal{X}$; see \eqref{eq:Vert_IC_Def} below. Then there exists a unique $\rho \in \mathbf{L}^{\infty}([0,\infty) \times \mathbb{R})$ satisfying \eqref{eq:sol_low_upp_bounds} that is a weak solution to \eqref{eq:model_nonlocal_1}-\eqref{eq:model_nonlocal_3} with initial condition $\rho(0,x) = \rho_0(x)$; 
  in other words,
  $\rho$ satisfies
  \begin{equation}\label{eq:WeakSolnDef}
    \int_0^{\infty} \int_{\mathbb{R}} \rho \partial_t \varphi + \rho v(q) \partial_x \varphi \; dx dt + \int_{\mathbb{R}} \rho_0(x) \varphi(0,x) dx = 0
  \end{equation}
  for all $\varphi \in \mathbf{C}^1_{\mathrm{c}}([0,\infty) \times \mathbb{R})$ with $q$ defined by \eqref{eq:model_nonlocal_3}.

  Additionally, for any $T >0$ there exists a constant $C = C(\gamma,v,w,T)$ such that the following holds: Let $\rho_0^i(x)$, $i=1,2$, belong to $\mathcal{X}$ with  $\rho_0^1 - \rho_0^2 \in \mathbf{L}^1(\mathbb{R})$, and let
  $(\rho^i,q^i)$ denote the solution pairs associated to \eqref{eq:model_nonlocal_1}-\eqref{eq:model_nonlocal_3} with initial condition $\rho^i(0,x) = \rho_0^i(x)$. Then
  \begin{equation}
    \int_0^T  \int_{\mathbb{R}} |\rho^1(t,x)-\rho^2(t,x)| dx dt  \leq C (1 + \mathrm{TV}(\rho_0^1) + \mathrm{TV}(\rho_0^2) ) \int_{\mathbb{R}} |\rho_0^1(x)-\rho_0^2(x)| dx.
  \end{equation}
\end{theorem}
The key tool used to prove Theorem \ref{thm:WeakSolnExistence} is the $\mathbf{L}^1$-stability of Lipschitz solutions; once that is established in Proposition \ref{prop:L1Stab:1}, the existence and uniqueness of weak solutions follows by using Theorem \ref{thm:wellposedness_lipschitz} and an approximation argument.

With the well-posedness of the problem \eqref{eq:model_nonlocal_1}-\eqref{eq:model_nonlocal_3} in hand, we analyze the nonlocal-to-local limit.
This limit is realized in the following way: consider the rescaled kernels
$w_\ep(s) = \ep^{-1} w(s/\ep)$.
Taking $\ep\to 0$, the kernel $w_\ep$ converges to a Dirac delta function,
and so -- formally -- solutions of the nonlocal model
\eqref{eq:model_nonlocal_1}-\eqref{eq:model_nonlocal_3} converge to the entropy admissible solution of the local model \eqref{eq:local_model}.
We make the choice of exponential kernel function for $w$, defined as
\begin{align}\label{eq:exp_kernel}
  w(s) = e^{-s}, \qquad
  w_\ep(s) = \ep^{-1} w(s/\ep) = \ep^{-1} e^{-s/\ep},\quad s\in[0,\infty).
\end{align}
With $w$ defined as in \eqref{eq:exp_kernel}, the model \eqref{eq:model_nonlocal_1}-\eqref{eq:model_nonlocal_3} (and more generally \eqref{eq:model_nonlocal_1}-\eqref{eq:model_nonlocal_2})
can be rewritten as a relaxation system:
\begin{align}
  \partial_t \rho + \partial_x (\rho v(q)) &=0, \label{model_relax_1}\\
  \partial_t q - \gamma^{-1}\partial_x q &= (\gamma\ep)^{-1} (\rho - q). \label{model_relax_2}
\end{align}
Utilizing the special features of this relaxation system formulation
\eqref{model_relax_1}-\eqref{model_relax_2},
a uniform global BV bound on $\rho$ that is independent of the relaxation parameter $\ep$
can be proved, which serves as a key estimate for the compactness theory and guarantees the existence of a limit of the solutions.

\begin{theorem}\label{thm:BV_bound}
  Suppose that Assumption 1, Assumption 2 and \eqref{eq:GammaSmallAssump} are satisfied, and let $\rho_0 \in \mathcal{X}$. Assume that the weight kernel is given by the exponential functions as in \eqref{eq:exp_kernel}.
  In addition, suppose that the minimum density $\rho_{\mathrm{min}}$ as defined in \eqref{eq:sol_low_upp_bounds} is positive, and that the following condition holds for $\gamma$ and $v$:
  \begin{align}\label{eq:BV_condition_gamma}
    \left( 1 - 2\gamma \norm{v'}_{\infty} \right) \min_{\rho\in[0,1]} |v'(\rho)| \geq (1 + \gamma v_{\mathrm{max}}) \norm{v''}_{\infty}.
  \end{align}
  Then the unique weak solution of \eqref{eq:model_nonlocal_1}-\eqref{eq:model_nonlocal_3} with initial condition $\rho(0,x) = \rho_0(x)$ satisfies
  \begin{align}\label{eq:uniform_BV_bound}
    \mathrm{TV}(\rho; [0,T]\times\mathbb{R}) \leq C T (1 + \rho_{\mathrm{min}}^{-1})  \mathrm{TV}(\rho_0) \quad \forall T>0,
  \end{align}
  where $\mathrm{TV}(\rho; [0,T]\times\mathbb{R})$ represents the total variation of $\rho$ on $[0,T]\times\mathbb{R}$, and the constant $C=C\left(\gamma,v\right)$ is independent of $\ep$.
\end{theorem}
The choice \eqref{eq:exp_kernel} is the same as the one made in \cite{bressan2019traffic} to analyze the nonlocal-to-local limit for the nonlocal-in-space model \eqref{eq:model_nonlocal_space_1}-\eqref{eq:model_nonlocal_space_2}. Our methods closely follow theirs, but the relaxation system \eqref{model_relax_1}-\eqref{model_relax_2} is a genuine system of conservation laws in the original $(t,x)$-coordinate system, and we additionally take this into account.
We remark that, in the case of $\gamma=0$, the condition \eqref{eq:GammaSmallAssump} holds whenever $w'(s)\leq0 ~\forall s\in[0,+\infty)$, and \eqref{eq:BV_condition_gamma} becomes $\min_{\rho\in[0,1]} |v'(\rho)| \geq \norm{v''}_{\infty}$. These conditions on the functions $w,v$ are the same as the ones proposed in \cite{bressan2019traffic} for the nonlocal-in-space model \eqref{eq:model_nonlocal_space_1}-\eqref{eq:model_nonlocal_space_2}.

Finally, we show that any limit solution of the space-time nonlocal model \eqref{eq:model_nonlocal_1}-\eqref{eq:model_nonlocal_3} when $\ep\to0$ is the unique weak entropy solution of \eqref{eq:local_model}.

\begin{theorem}\label{thm:limit_sol_entropy}
  Under the same assumptions as in Theorem~\ref{thm:BV_bound}, let $\rho^\ep$ be the unique weak solution of \eqref{eq:model_nonlocal_1}-\eqref{eq:model_nonlocal_3} with initial condition $\rho^{\varepsilon}(0,x) = \rho_0(x)$ as in Theorem~\ref{thm:WeakSolnExistence}. Then the solution $\rho^\ep$ converges to the unique entropy solution of \eqref{eq:local_model} in $\mathbf{L}^1_{\mathrm{loc}}([0,\infty)\times\mathbb{R})$ as $\ep\to0$.
\end{theorem}

\subsection{Organization of the paper}\label{sub:organization}
This paper is organized as follows. First, we establish the existence of Lipschitz solutions from Lipschitz past-time data in Section~\ref{sec:existence_lip_sol} (Theorem~\ref{thm:wellposedness_lipschitz}). In Section \ref{sec:continuous_dependence} we establish the $\mathbf{L}^1$ stability estimate for Lipschitz solutions and prove Theorem \ref{thm:WeakSolnExistence}.
Section~\ref{sec:BV_bound} is devoted to the uniform BV bound estimate of solutions based on the model's relaxation system formulation (Theorem~\ref{thm:BV_bound}), which guarantees the existence of local limit solutions. Section~\ref{sec:entropy_admissbility} provides the proof of entropy admissibility of the local limit solution and completes the nonlocal-to-local limit theorem (Theorem~\ref{thm:limit_sol_entropy}).

\section{Existence of Lipschitz solutions}
\label{sec:existence_lip_sol}

This section is devoted to the proof of Theorem~\ref{thm:wellposedness_lipschitz}.

\subsection{Initial and past-time data}\label{sub:data}
To begin, we make precise the conditions on the past-time data.
First, the initial values of $\rho$ and $q$ corresponding to a past-time condition $\rho_-$ are denoted throughout the paper as
\begin{equation}\label{eq:rho_q_zero_def}
  \rho_0(x) \doteq \rho_-(0,x), \quad q_0(x) \doteq \int_0^\infty \rho_{-}(-\gamma s, x+s) w(s) \, ds, \qquad x\in\mathbb{R}.
\end{equation}
Second, for a given constant $L > 0$
we introduce the following notation for a class of functions for past-time data $\rho_-$ with $\rho_0$ and $q_0$ given by
\eqref{eq:rho_q_zero_def} correspondingly.
{\small \begin{equation}\label{eq:Lip_IC_Def}
\begin{split}
&\mathcal{X}_{\mathrm{Lip},L} \\&\doteq \left\{
\begin{gathered}
\rho_- \in \mathbf{L}^{\infty}((-\infty,0] \times \mathbb{R}) : \,
  \inf_{(t,x)\in(-\infty,0]\times\mathbb{R}}\rho_{-}(t,x) > 0, \; \sup_{(t,x)\in(-\infty,0]\times\mathbb{R}}\rho_{-}(t,x) < 1, \\
\qquad    \inf_{x\in\mathbb{R}}\rho_0(x) (1 + \gamma v(q_0(x))) > 0, \quad \sup_{x\in\mathbb{R}}\rho_0(x) (1 + \gamma v(q_0(x))) < 1, \\
\qquad    \sup_{(t,x)\in(-\infty,0]\times\mathbb{R}}|(\partial_x-\gamma \partial_t) \rho_{-}(t,x)| \leq L, \quad \sup_{x\in\mathbb{R}} | \partial_x(\rho_0(x)(1+\gamma v(q_0(x)))) | \leq L
\end{gathered}
\right\}.
\end{split}
\end{equation}}
Now we define
\begin{equation}\label{eq:g-def}
g(\rho) \doteq \rho (1+\gamma v(\rho)), \quad \rho\in[0,1].
\end{equation}
Under the Assumption 1, we have $g(0)=0$ and $g(1)=1$.
Moreover, it holds that $g'(\rho)>0$ for $\rho\in[0,1]$ provided $\gamma \norm{v'}_\infty < 1$.
In this case the function $g$ is monotone and we let $g^{-1}$ denote the inverse function of $g$. We define
\begin{align}
  & \rho_{\mathrm{min}} \doteq \min \left\{ \inf_{(t,x)\in(-\infty,0]\times\mathbb{R}}\rho_{-}(t,x), ~g^{-1}\left(\inf_{x\in\mathbb{R}}\rho_0(x)(1+\gamma v(q_0(x))\right) \right\}, \label{eq:sol_lower_bound} \\
  & \rho_{\mathrm{max}} \doteq \max \left\{ \sup_{(t,x)\in(-\infty,0]\times\mathbb{R}}\rho_{-}(t,x), ~g^{-1}\left(\sup_{x\in\mathbb{R}}\rho_0(x)(1+\gamma v(q_0(x))\right) \right\},\label{eq:sol_upper_bound}
\end{align}
where $\rho_0$ and $q_0$ are defined in \eqref{eq:rho_q_zero_def}. It is clear that $0 < \rho_{\mathrm{min}} \leq \rho_{\mathrm{max}} < 1$ for any $\rho_- \in \mathcal{X}_{\mathrm{Lip},L}$.

\subsection{Reformulation as a fixed-point problem}\label{sub:reformulation}

The essential idea in the proof of Theorem \ref{thm:wellposedness_lipschitz} is to reformulate the model as a fixed point problem and apply the contraction mapping theorem.
We first define the fixed point mapping on a proper domain with a finite time horizon, and then show it is contractive through a priori $\mathbf{L}^\infty$ and Lipschitz estimates.
The fixed point solution is shown to be a Lipschitz solution to the model and it can be extended to all times $t>0$.

First let us fix a time horizon $[0,T]$ where $T>0$, and suppose $\rho_{\mathrm{min}}, \rho_{\mathrm{max}}$,
are as defined in \eqref{eq:sol_lower_bound}-\eqref{eq:sol_upper_bound}.
For any
$0 \leq \rho_a < \rho_{\mathrm{min}}$ and $\rho_{\mathrm{max}} < \rho_b \leq 1$, we define the domain
\begin{equation*}
  \mathcal{D}_{T, L, \rho_a, \rho_b}
  \doteq  \left\{ \rho\in\mathbf{L}^\infty ([0,T]\times\mathbb{R}) \, : \,
  \begin{gathered}
    \rho_a \leq \rho(t,x) \leq \rho_b, \; (t,x) \in [0,T]\times\mathbb{R}, \\
    |(\partial_x-\gamma \partial_t)\rho(t,x)|\leq 3L, \; (t,x) \in (0,T)\times\mathbb{R}, \\
        \rho(0,x) = \rho_0(x), \; x \in \mathbb{R}
  \end{gathered}
  \right\}.
\end{equation*}

Then we introduce a directional derivative operator
\[
  \partial_y \doteq \partial_x-\gamma\partial_t,
\]
where the direction is taken along the line integral paths in \eqref{eq:model_nonlocal_2}, and an auxiliary variable
\[
    z \doteq \rho (1+\gamma v(q)).
\]
With the above definitions, the past-time value problem \eqref{eq:model_nonlocal_1}-\eqref{eq:model_nonlocal_2}-\eqref{eq:initial_negative} can be reformulated as a system to be solved on $[0,T]\times\mathbb{R}$.
\begin{align}
        & q(t,x) = \int_0^{t/\gamma} \rho(t-\gamma s, x+s) w(s)\, ds + \int_{t/\gamma}^\infty \rho_-(t-\gamma s, x+s) w(s)\, ds, \label{eq:model_z_q}\\
        & z(t,x) = \rho(t,x) (1+\gamma v(q(t,x))), \label{eq:def_z} \\
        & \partial_t z(t,x) + \partial_y \left( \frac{v(q(t,x))}{1+\gamma v(q(t,x))} z(t,x) \right) = 0. \label{eq:model_z}
\end{align}
This representation motivates the following step-by-step definition of a mapping $\Gamma:\, \mathcal{D}_{T, L, \rho_a, \rho_b} \to  \mathbf{L}^\infty([0,T]\times\mathbb{R}) $.
\begin{enumerate}
\item With a given $\rho_- \in \mathcal{X}_{\mathrm{Lip},L}$ and for any $\rho \in \mathcal{D}_{T, L, \rho_a, \rho_b}$, we define $q(t,x;\rho,\rho_-)$ for all $(t,x)\in[0,T]\times\mathbb{R}$ as in \eqref{eq:model_z_q}.
\item We define $z(t,x;\rho,\rho_-)$ for all $(t,x)\in[0,T]\times\mathbb{R}$ as the solution to the linear Cauchy problem \eqref{eq:model_z} with the above $q(t,x;\rho,\rho_-)$ and the initial condition
$$ z(0,x;\rho_-) = \rho_0(x) (1+\gamma v(q_0(x))). $$
\item
With $z(t,x;\rho,\rho_-)$ and $q(t,x;\rho,\rho_-)$ defined above, we define $\tilde{\rho}(t,x;\rho,\rho_-)$ as
\begin{align*}
  \tilde{\rho}(t,x;\rho,\rho_-) = \frac{z(t,x;\rho,\rho_-)}{1+\gamma v(q(t,x;\rho,\rho_-))}, \quad (t,x) \in [0,T]\times\mathbb{R}.
\end{align*}
\end{enumerate}
Finally we define the mapping $\Gamma$ by
$$ \Gamma[\rho](t,x) \doteq \tilde{\rho}(t,x;\rho,\rho_-), \quad (t,x) \in [0,T]\times\mathbb{R}. $$

The outline of the proof of Theorem~\ref{thm:wellposedness_lipschitz} is
to establish the following facts:
\begin{itemize}
\item For any $\rho \in \mathcal{D}_{T, L, \rho_a, \rho_b}$, $\Gamma[\rho] \in \mathcal{D}_{T, L, \rho_a, \rho_b}$;
\item $\Gamma$ is a contraction mapping on $\mathcal{D}_{T, L, \rho_a, \rho_b}$ in the $\mathbf{L}^\infty$ norm;
\item The contraction mapping theorem gives the unique fixed point $\rho \in \mathcal{D}_{T, L, \rho_a, \rho_b}$, i.e.\ $\Gamma[\rho]=\rho$;
\item The fixed point solution is Lipschitz and it solves the system \eqref{eq:model_z_q}-\eqref{eq:def_z}-\eqref{eq:model_z} for $t\in[0,T]$;
\item By continuation, the constructed solution for $t\in[0,T]$ can be extended to $t\in[0,\infty)$ and so it solves the past-time value problem \eqref{eq:model_nonlocal_1}-\eqref{eq:model_nonlocal_2}-\eqref{eq:initial_negative}.
\end{itemize}

We remark here that the map $\Gamma$ as constructed requires no relation between $\rho$ and $\rho_-$ at $t=0$ to hold. However, the condition $\rho(0,x)=\rho_0(x)$ is imposed so that quantities such as $q(t,x)$ are Lipschitz with appropriate constant.

\subsection{Proof of Theorem~\ref{thm:wellposedness_lipschitz}}\label{sub:proof1}
The proof consists of
six steps. In the proof, we 
omit the notations $\rho$ and $\rho_-$ in $q(t,x;\rho,\rho_-)$, $z(t,x;\rho,\rho_-)$ and $\tilde{\rho}(t,x;\rho,\rho_-)$ for simplicity, but keep in mind that they both depend on $\rho$ and $\rho_-$.
In addition, we use the equation \eqref{eq:model_nonlocal_2} for $q$ to simplify the calculation, but keep in mind that the nonlocal integral for $q$ involves $\rho_-$ and its precise form is \eqref{eq:model_z_q}.

\vspace{4pt}\noindent
\textbf{Step 1 (Characteristics).} We rewrite the linear Cauchy problem \eqref{eq:model_z} as
\begin{align}\label{eq:model_z_transport}
  \partial_t z + \frac{v(q)}{1+\gamma v(q)} \partial_y z = z \frac{-v'(q)}{(1+\gamma v(q))^2} \partial_y q.
\end{align}
Given $z(0,x)$ for $x\in\mathbb{R}$ and $q(t,x)$ for $(t,x)\in[0,T]\times\mathbb{R}$, \eqref{eq:model_z_transport} can be solved by the method of characteristics. For a point $(t,x)$, the characteristic curve is given by $\tau \mapsto (\tau - \gamma \xi(\tau),\xi(\tau))$ where $\xi(\tau)$ satisfies
\begin{align}\label{eq:char}
    \frac{d\xi(\tau)}{d\tau} = \frac{v(q(\tau-\gamma\xi(\tau), \xi(\tau)))}{1+\gamma v(q(\tau-\gamma\xi(\tau), \xi(\tau)))},
    \qquad
    \xi(t+\gamma x) = x, \qquad \tau \in \mathbb{R}\,.
\end{align}
It is easy to see that by definition of $\rho_{\mathrm{min}}$ and $\rho_{\max}$ that
\begin{align}\label{eq:q_bounds}
  0 \leq \rho_a \leq q(t,x) \leq \rho_b \leq 1, \quad (t,x) \in [0,T]\times\mathbb{R}.
\end{align}
This implies
\[
0\leq \frac{d\xi}{d \tau} (\tau) \leq v_{\mathrm{max}}\qquad\mbox{ and }\qquad
\frac{d}{d \tau} [\xi - \gamma \xi(\tau) ] \geq1-\gamma v_{\mathrm{max}}>0\]
for all characteristic curves. Therefore, for any given point $(t,x)\in[0,T]\times\mathbb{R}$ one can trace the characteristic curve back to reach a unique point $(0,x')$ on the $x$-axis, and $x' - \frac{t}{\gamma} \leq x' \leq x$.
Integrating the characteristic ODE
\begin{align}\label{eq:char_eq}
&  \frac{d}{d\tau}z(\tau-\gamma\xi(\tau), \xi(\tau))\notag\\&=z(\tau-\gamma\xi(\tau), \xi(\tau))\frac{-v'(q(\tau-\gamma\xi(\tau), \xi(\tau)))}{(1+\gamma v(q(\tau-\gamma\xi(\tau), \xi(\tau))))^2}\cdot \partial_y q(\tau-\gamma\xi(\tau), \xi(\tau)),
\end{align}
from the unique $\tau_0$ satisfying $(\tau_0 - \gamma \xi(\tau_0),\xi(\tau_0)) = (0,x')$ to $\tau_1=t+\gamma x$, one can obtain the value of $z(t,x)$.

\vspace{4pt}\noindent
\textbf{Step 2 ($\mathbf{L}^\infty$ and directional Lipschitz bounds).}
We first note that
the identity
\begin{align*}
  \partial_y q(t, x)& = \int_0^{t/\gamma} \partial_y \rho(t-\gamma s, x+s) w(s)\, ds + \int_{t/\gamma}^\infty \partial_y \rho_-(t-\gamma s, x+s) w(s)\, ds,\\  &\qquad  (t,x)\in[0,T]\times\mathbb{R},
\end{align*}
gives
$$
|\partial_y q(t,x)| \leq 3L
$$ for all $(t,x)\in[0,T]\times\mathbb{R}$. In addition, integration by parts gives
\begin{align}\label{eq:estim_D2q}
  \partial_{yy} q(t, x) = -w(0)\partial_y\rho(t,x) - \int_0^\infty \partial_y\rho(t-\gamma s, x+s) w'(s)\,ds,
\end{align}
hence
$$
|\partial_{yy} q(t,x)| \leq 6 w(0)L
$$ for all $(t,x)\in[0,T]\times\mathbb{R}$.

To give a $\mathbf{L}^\infty$ bound on $\tilde{\rho}$, we note that
\[
    g(\rho_a) < g(\rho_{\mathrm{min}}) \leq z(0,x) \leq g(\rho_{\mathrm{max}}) < g(\rho_b), \quad x\in\mathbb{R}.
\]
By integrating \eqref{eq:char_eq} and using the uniform bound
\begin{align*}
    \norm{ \frac{-v'(q)}{(1+\gamma v(q))^2}\partial_y q }_\infty &\leq 3 \norm{v'}_\infty L,
\end{align*}
we obtain that
\begin{align*}
    g(\rho_a) \leq z(t,x) \leq g(\rho_b), \quad (t,x)\in[0,T]\times\mathbb{R},
\end{align*}
when $T$ is sufficiently small. This together with \eqref{eq:q_bounds} gives
\begin{align}\label{eq:tilde_rho_Linf_bound}
    \rho_a \leq \tilde{\rho}(t,x) \leq \rho_b, \quad (t,x)\in[0,T]\times\mathbb{R}.
\end{align}

Now let us give a bound on $\partial_y\tilde{\rho}$. Taking the directional derivative $\partial_y$ of the equation \eqref{eq:model_z_transport}, we obtain
\begin{align}\label{eq:char_eq_lip}
  \partial_t(\partial_y z) + &\frac{v(q)}{1+\gamma v(q)} \partial_y(\partial_y z) = \partial_y z \frac{-2v'(q)}{(1+\gamma v(q))^2} \partial_y q \notag\\
  &+ z \left[ \frac{2\gamma(v'(q))^2-v''(q)(1+\gamma v(q))}{(1+\gamma v(q))^3}(\partial_y q)^2 + \frac{-v'(q)}{(1+\gamma v(q))^2}\partial_{yy} q \right].
\end{align}
At time $t=0$, by the equation \eqref{eq:model_z_transport} we write
\begin{align*}
    \partial_y z(0,x) = (1+\gamma v(q(0,x))) \partial_x z(0,x) + z(0,x) \frac{\gamma v'(q(0,x))}{1 + \gamma v(q(0,x))} \partial_y q(0,x),
\end{align*}
using that $\rho_- \in \mathcal{X}_{\mathrm{Lip},L}$ we have
\begin{align*}
  |\partial_y z(0,x)| \leq \left( 1 + \gamma v_{\mathrm{max}} + \gamma \norm{v'}_\infty \right) L \leq \frac43 L,\quad x\in\mathbb{R}.
\end{align*}
We integrate the equation \eqref{eq:char_eq_lip} along the characteristic curves defined in \eqref{eq:char}. With the uniform bounds
\begin{align*}
    \norm{ \frac{-2v'(q)}{(1+\gamma v(q))^2}\partial_y q }_\infty &\leq 6 \norm{v'}_\infty L, \\
  \norm{ z \frac{2\gamma(v'(q))^2-v''(q)(1+\gamma v(q))}{(1+\gamma v(q))^3} (\partial_y q)^2 }_\infty &\leq \left( 2\gamma\norm{v'}_\infty^2+\norm{v''}_\infty \right) 9 L^2 \cdot g(\rho_b), \\
  \norm{ z\frac{-v'(q)}{(1+\gamma v(q))^2}\partial_{yy} q  }_\infty & \leq 6 w(0)\norm{v'}_\infty L \cdot  g(\rho_b),
\end{align*}
we deduce from a comparison argument that
\begin{align}\label{eq:lip_bound}
    \sup\nolimits_{x\in\mathbb{R}}|\partial_y z(t,x)| \leq Z(t), \quad t\in[0,T],
\end{align}
where $Z(t)$ is the solution to the linear ODE
\begin{align}\label{eq:lip_ode}
  \dot{Z} = a Z + b, \quad Z(0) = \frac43 L,
\end{align}
with constant coefficients given by
\begin{align*}
a \doteq \frac{6 \norm{v'}_\infty L}{1-\gamma v_{\mathrm{max}}}, \qquad
b \doteq \frac{9 \left( 2\gamma\norm{v'}_\infty^2+\norm{v''}_\infty \right) L^2 + 6w(0) \norm{v'}_\infty L}{1-\gamma v_{\mathrm{max}}}.
\end{align*}
By choosing $T$ sufficiently small, we obtain that $|\partial_y z(t,x)| \leq Z(T) \leq 2 L$ for $(t,x)\in[0,T]\times\mathbb{R}$. Then the identity
\begin{align*}
 & \partial_y \tilde{\rho}(t,x)\\& = \frac{(1+\gamma v(q(t,x))) \partial_y z(t,x)-\gamma z(t,x) v'(q(t,x)) \partial_y q(t,x)}{(1+\gamma v(q(t,x)))^2}, \quad (t,x)\in[0,T]\times\mathbb{R},
\end{align*}
implies that
\begin{align}\label{eq:tilde_rho_lip_bound}
  |\partial_y\tilde{\rho}(t,x)| \leq 2L + 3 \gamma\norm{v'}_\infty L \leq 3 L, \quad (t,x)\in[0,T]\times\mathbb{R}.
\end{align}

The equality $\tilde{\rho}(0,x) = \rho_0(x)$ is clear from the definition. Using the obtained $\mathbf{L}^\infty$ and directional Lipschitz bounds \eqref{eq:tilde_rho_Linf_bound}-\eqref{eq:tilde_rho_lip_bound}, we conclude that there exist $L,T>0$, depending only on $\gamma,v,L_0,\rho_{\mathrm{min}}, \rho_{\mathrm{max}}$, such that $\Gamma$ maps $\mathcal{D}_{T, L, \rho_a, \rho_b}$ to itself.
 
\vspace{4pt}\noindent 
\textbf{Step 3 (Contraction).}
For any $\rho_1,\rho_2 \in \mathcal{D}_{T, L, \rho_a, \rho_b}$,
and any $(t,x)\in[0,T]\times\mathbb{R}$, we denote by
\begin{align*}
  \begin{cases}
    t_1(\tau)=\tau-\gamma\xi_1(\tau) \\ x_1(\tau)=\xi_1(\tau)
  \end{cases} \text{and} \quad
  \begin{cases}
    t_2(\tau)=\tau-\gamma\xi_2(\tau) \\ x_2(\tau)=\xi_2(\tau)
  \end{cases} ,
\end{align*}
the two characteristic curves satisfying
\begin{align*}
  \frac{d\xi_i(\tau)}{d\tau} = \frac{v(q_i(\tau-\gamma\xi_i(\tau), \xi_i(\tau)))}{1+\gamma v(q_i(\tau-\gamma\xi_i(\tau), \xi_i(\tau)))}, \quad \xi_i(t+\gamma x)=x, \quad i=1,2.
\end{align*}
We have
\begin{align*}
  -\frac{d|\xi_1(\tau)-\xi_2(\tau)|}{d\tau} &\leq \norm{v'}_\infty \left| q_1(\tau-\gamma\xi_1(\tau), \xi_1(\tau)) - q_2(\tau-\gamma\xi_2(\tau), \xi_2(\tau)) \right|, \\
  &\leq \norm{v'}_\infty \left( \norm{\partial_y q_1}_\infty|\xi_1(\tau)-\xi_2(\tau)| + \norm{q_1-q_2}_\infty \right), \\
  &\leq \norm{v'}_\infty \left( 3L |\xi_1(\tau)-\xi_2(\tau)| + \norm{q_1-q_2}_\infty \right).
\end{align*}
Using Gr\"onwall's inequality backward in time, we obtain
\begin{align}\label{eq:contraction_xi_difference}
  |\xi_1(\tau)-\xi_2(\tau)| \leq C_0T \norm{q_1-q_2}_\infty, \quad 0\leq t_1(\tau), t_2(\tau)\leq t,
\end{align}
with the constant $C_0=C_0(L,v)>0$.

Note that $z_1$ and $z_2$ can be solved from
\begin{align*}
  \partial_tz_i + \frac{v(q_i)}{1+\gamma v(q_i)} \partial_y z_i=z_i\frac{-v'(q_i)}{(1+\gamma v(q_i))^2} \partial_y q_i, \quad i=1,2,
\end{align*}
along the characteristic curves $(t_i(\tau),x_i(\tau))$, $i=1,2$ with the same initial condition. Using again Gr\"onwall's inequality and noticing \eqref{eq:contraction_xi_difference}, we obtain
\begin{align*}
  \norm{z_1-z_2}_\infty\leq C_1 T \norm{q_1-q_2}_\infty,
\end{align*}
with the constant $C_1=C_1\left(L,\gamma,v,\norm{\partial_y^2q_1}_\infty,\norm{\partial_y^2q_2}_\infty\right)>0$.

We have
\begin{align*}
 & \tilde{\rho}_1(t,x) - \tilde{\rho}_2(t,x)\\& = \frac{z_1(t,x)(1+\gamma v(q_2(t,x)))-z_2(t,x)(1+\gamma v(q_1(t,x)))}{(1+\gamma v(q_1(t,x)))(1+\gamma v(q_2(t,x)))}, \quad (t,x)\in [0,T]\times\mathbb{R},
\end{align*}
which implies
\begin{align*}
  \norm{\tilde{\rho}_1-\tilde{\rho}_2}_\infty \leq \gamma \norm{v'}_\infty\norm{q_1-q_2}_\infty + \norm{z_1-z_2}_\infty.
\end{align*}
We also have
\begin{align*}
  q_1(t,x)-q_2(t,x) = \int_0^{t/\gamma} (\rho_1(t-\gamma s,x+s)-\rho_2(t-\gamma s,x+s)) w(s)\,ds,
\end{align*}
which gives
$$
\norm{q_1-q_2}_\infty \leq \frac{w(0)T}{\gamma} \norm{\rho_1-\rho_2}_\infty.
$$
Apply \eqref{eq:estim_D2q} to both $\partial_{yy} q_1$ and $\partial_{yy} q_2$, one can get
\begin{align*}
  \norm{\partial_{yy} q_i}_\infty
  \leq 6w(0)L, \quad i=1,2.
\end{align*}

Thanks to the above estimates, we finally deduce that
\begin{align*}
  \norm{\Gamma[\rho_1]-\Gamma[\rho_2]}_\infty = \norm{\tilde{\rho}_1 - \tilde{\rho}_2}_\infty \leq C_2T\norm{\rho_1-\rho_2}_\infty,
\end{align*}
with the constant $C_2=C_2(L,\gamma,v,w)>0$. Choosing $T$ sufficiently small such that $C_2T<1$, $\Gamma$ is a contraction mapping in the $\mathbf{L}^\infty$ norm.

By the contraction mapping theorem, there exists $T^* >0 $ such that $\Gamma$ has a unique fixed point in $\mathcal{D}_{T^*, L, \rho_a, \rho_b}$. From now on we denote $\rho$ as the unique solution in $\mathcal{D}_{T^*,L,\rho_a,\rho_b}$ that satisfies   \eqref{eq:model_nonlocal_1}-\eqref{eq:model_nonlocal_2}-\eqref{eq:initial_negative} on $[0,T^*] \times \mathbb{R}$. With this definition of $\rho$ we define $z$ by \eqref{eq:def_z}.

\vspace{4pt}\noindent
\textbf{Step 4 (Uniform $\mathbf{L}^\infty$ bound).}
We aim to show that $\rho$ and $z$ satisfy the uniform bounds
\begin{align}\label{eq:sol_uniform_bound}
  \rho_{\mathrm{min}} \leq \rho(t,x) \leq \rho_{\mathrm{max}} \quad \mathrm{and} \quad g(\rho_{\mathrm{min}}) \leq z(t,x) \leq g(\rho_{\mathrm{max}}), \quad (t,x)\in[0,T^*]\times\mathbb{R}.
\end{align}
We provide a proof for the upper bounds; the lower bounds are obtained in a similar manner.

We denote
\begin{align*}
    \rho_{\mathrm{m}} \doteq \sup_{(t,x)\in (-\infty,T^*]\times\mathbb{R}} \rho(t,x), \qquad z_{\mathrm{m}} \doteq \sup_{(t,x)\in[0,T^*]\times\mathbb{R}} z(t,x).
\end{align*}
It is clear that $\rho_{\mathrm{max}} \leq \rho_{\mathrm{m}} \leq 1$ and $g(\rho_{\mathrm{max}}) \leq z_{\mathrm{m}} \leq 1$.

Let us fix $x_0\in\mathbb{R}$ and consider the characteristic curve $\tau \mapsto (\tau - \gamma \xi(\tau),\xi(\tau))$ for $\tau\in[\tau_0,\tau_1]$ such that
\begin{equation*}
    (\tau_0 - \gamma \xi(\tau_0),\xi(\tau_0)) = (0,x_0), \quad (\tau_1 - \gamma \xi(\tau_1),\xi(\tau_1)) = (T^*,x_1),
\end{equation*}
where $(T^*,x_1)$ is the intersection of the characteristic curve and the horizontal line $t=T^*$.
For any $\tau\in[\tau_0,\tau_1]$, the equation \eqref{eq:model_z_transport} gives
\begin{align*}
  \frac{d}{d\tau} z(\tau - \gamma \xi(\tau),\xi(\tau)) = \left. z \frac{-v'(q)}{(1+\gamma v(q))^2} \partial_y q \right|_{(\tau - \gamma \xi(\tau),\xi(\tau))} .
\end{align*}
Integrating by parts gives
\begin{align}
  &\partial_y q(\tau - \gamma \xi(\tau),\xi(\tau))\notag \\&= \int_0^\infty \partial_y\rho(\tau - \gamma \xi(\tau) -\gamma s, \xi(\tau) + s) w(s)\,ds \notag\\
  &= -w(0) \rho(\tau - \gamma \xi(\tau),\xi(\tau)) -\int_0^\infty\rho(\tau - \gamma \xi(\tau) -\gamma s, \xi(\tau) + s) w'(s)\,ds \notag\\
    &= w(0)\left[ \int_0^\infty \rho(\tau - \gamma \xi(\tau) -\gamma s, \xi(\tau) + s) \tilde{w}(s)\,ds - \rho(\tau - \gamma \xi(\tau),\xi(\tau)) \right], \label{eq:estim_Dq_2}
\end{align}
where
\[
\tilde{w}(s) \doteq -w'(s)/w(0)
\]
is a new weight kernel satisfying
\begin{align*}
\int_0^\infty \tilde{w}(s)\,ds=1 \quad \mbox{and}\quad
  \tilde{w}(s)\geq \frac{\beta}{w(0)}w(s)\geq\gamma\norm{v'}_\infty w(s)\geq0 \quad \text{ for } s\in[0,\infty).
\end{align*}
Noting that
\begin{align*}
  v(q)=v(q)-v(\rho_{\mathrm{m}})+v(\rho_{\mathrm{m}}) \leq \norm{v'}_\infty(\rho_{\mathrm{m}}-q)+v(\rho_{\mathrm{m}}) \quad \forall q\in[0,\rho_{\mathrm{m}}],
\end{align*}
we compute
\begin{align*}
  & z(\tau - \gamma \xi(\tau),\xi(\tau)) \\
  =& \rho(\tau - \gamma \xi(\tau),\xi(\tau)) (1+\gamma v(q(\tau - \gamma \xi(\tau),\xi(\tau))))\\
  \leq& \rho(\tau - \gamma \xi(\tau),\xi(\tau)) + \gamma \rho_{\mathrm{m}} v(q(\tau - \gamma \xi(\tau),\xi(\tau)))\\
  \leq& \rho(\tau - \gamma \xi(\tau),\xi(\tau)) + \gamma \rho_{\mathrm{m}} \norm{v'}_\infty(\rho_{\mathrm{m}}-q(\tau - \gamma \xi(\tau),\xi(\tau))) + \gamma \rho_{\mathrm{m}} v(\rho_{\mathrm{m}})\\
  \leq& \rho(\tau - \gamma \xi(\tau),\xi(\tau)) \\&+ \gamma \norm{v'}_\infty\int_0^\infty (\rho_{\mathrm{m}}-\rho(\tau - \gamma \xi(\tau) -\gamma s, \xi(\tau) + s)) w(s)\,ds + \gamma \rho_{\mathrm{m}} v(\rho_{\mathrm{m}})\\
    \leq& \rho(\tau - \gamma \xi(\tau),\xi(\tau)) + \int_0^\infty (\rho_{\mathrm{m}}-\rho(\tau - \gamma \xi(\tau) -\gamma s, \xi(\tau) + s)) \tilde{w}(s)\,ds + \gamma \rho_{\mathrm{m}} v(\rho_{\mathrm{m}}) \\
  =& \rho(\tau - \gamma \xi(\tau),\xi(\tau)) - \int_0^\infty \rho(\tau - \gamma \xi(\tau) -\gamma s, \xi(\tau) + s) \tilde{w}(s)\,ds + g(\rho_{\mathrm{m}}).
\end{align*}
It yields that
{\small \begin{align*}
  \int_0^\infty \rho(\tau - \gamma \xi(\tau) -\gamma s, \xi(\tau) + s) \tilde{w}(s)\,ds - \rho(\tau - \gamma \xi(\tau),\xi(\tau)) \leq g(\rho_{\mathrm{m}}) - z(\tau - \gamma \xi(\tau),\xi(\tau)).
\end{align*}}
This inequality combined with  \eqref{eq:estim_Dq_2} gives
\[
\partial_y q(\tau - \gamma \xi(\tau),\xi(\tau))\leq w(0) (g(\rho_{\mathrm{m}}) - z(\tau - \gamma \xi(\tau),\xi(\tau))).\]
Furthermore, we have
\begin{align*}
  0\leq \left. z \frac{-v'(q)}{(1+\gamma v(q))^2} \right|_{(\tau - \gamma \xi(\tau),\xi(\tau))} \leq \norm{v'}_\infty,
\end{align*}
and hence
\begin{align*}
  \frac{d}{d\tau}z(\tau-\gamma\xi(\tau), \xi(\tau)) \leq C (g(\rho_{\mathrm{m}}) - z(\tau - \gamma \xi(\tau),\xi(\tau))),
\end{align*}
where $C=\norm{v'}_\infty w(0)$.
Integrating the above inequality with the initial condition $z(\tau_0-\gamma\xi(\tau_0), \xi(\tau_0)) = z(0,x_0)$, we obtain that
\begin{align*}
    z(\tau-\gamma\xi(\tau), \xi(\tau)) \leq& e^{C(\tau_0-\tau)} z(0,x_0) + (1-e^{C(\tau_0-\tau)}) g(\rho_{\mathrm{m}}) \\
    \leq& e^{C(\tau_0-\tau)} g(\rho_{\mathrm{max}}) + (1-e^{C(\tau_0-\tau)}) g(\rho_{\mathrm{m}}), \quad \tau\in[\tau_0,\tau_1].
\end{align*}
Noting that $\tau_1-\tau_0\leq \frac{T^*}{1-\gamma v_{\mathrm{max}}}$ and $g(\rho_{\mathrm{max}}) \leq g(\rho_{\mathrm{m}})$, we have
\begin{align}\label{eq:char_uniform_bound}
    z(\tau-\gamma\xi(\tau), \xi(\tau)) \leq C_1 g(\rho_{\mathrm{max}}) + (1-C_1) g(\rho_{\mathrm{m}}) , \quad \tau\in[\tau_0,\tau_1],
\end{align}
where $C_1 = \mathrm{exp} (- \frac{\Vert v' \Vert_{\infty} w(0) T^* }{1- \gamma v_{\max}})  \in (0,1)$.
Now we let $x_0$ run over $\mathbb{R}$; the respective characteristic curves fill the domain $[0,T^*]\times\mathbb{R}$ and so \eqref{eq:char_uniform_bound} is uniform to the choice of $x_0$, hence we have
\begin{align}\label{eq:AprioriUniformBoundz}
    z_{\mathrm{m}} \leq C_1 g(\rho_{\mathrm{max}}) + (1-C_1) g(\rho_{\mathrm{m}}).
\end{align}

Now suppose $\rho_{\mathrm{m}}>\rho_{\mathrm{max}}$. Then we have
\begin{align*}
    \sup_{(t,x)\in[0,T^*]\times\mathbb{R}} q(t,x) \leq \rho_{\mathrm{m}} = \sup_{(t,x)\in[0,T^*]\times\mathbb{R}} \rho(t,x),
\end{align*}
and since $v$ is decreasing we have for any $(t,x)\in[0,T^*]\times\mathbb{R}$
\begin{align*}
    \rho(t,x) = \frac{z(t,x)}{1+\gamma v(q(t,x))} \leq \frac{z_{\mathrm{m}}}{1+\gamma v(\rho_{\mathrm{m}})}.
\end{align*}
Therefore
$\rho_{\mathrm{m}} \leq \frac{z_{\mathrm{m}}}{1+\gamma v(\rho_{\mathrm{m}})}$,
and so by definition of $g$ and by \eqref{eq:AprioriUniformBoundz}
\begin{align*}
    g(\rho_{\mathrm{m}}) \leq z_{\mathrm{m}} \leq C_1 g(\rho_{\mathrm{max}}) + (1-C_1) g(\rho_{\mathrm{m}}) \Longrightarrow g(\rho_{\mathrm{m}}) \leq g(\rho_{\mathrm{max}}),
\end{align*}
which contradicts $\rho_{\mathrm{m}}>\rho_{\mathrm{max}}$. Therefore we deduce that $\rho_{\mathrm{m}} \leq \rho_{\mathrm{max}}$. Applying this in \eqref{eq:AprioriUniformBoundz} gives $z_{\mathrm{m}} \leq g(\rho_{\mathrm{max}})$, and so the upper bounds in \eqref{eq:sol_uniform_bound} are proved.

\vspace{4pt}\noindent
\textbf{Step 5 (Final Lipschitz estimates).}
In Step 2, we obtain bounds on $\partial_y q$ and $\partial_y z$. Using the equation \eqref{eq:model_z}, a bound on $\partial_t z$ can also be obtained and we conclude that $z$ is Lipschitz continuous on $[0,T]\times\mathbb{R}$.
To show the Lipschitz continuity of $\rho$, it suffices to show that of $q$ 
since
$\rho=\frac{z}{1+\gamma v(q)}$. 
Given the established bound on $\partial_y q$,
we only need to show the existence of $\partial_x q$ and give a bound on it.

Let us denote
\begin{align*}
  K_0 \doteq \sup_{(t,x)\in(-\infty,0]\times\mathbb{R}} |\partial_x \rho_-(t,x)|, \quad K_1 \doteq \sup_{(t,x)\in[0,T]\times\mathbb{R}} |\partial_x z(t,x)|,
\end{align*}
and
\begin{align*}
  K(t,r) \doteq \sup_{|x_0-x_1|=r} \frac{|q(t,x_0)-q(t,x_1)|}{r} \quad \forall r>0, \, t\in[0,T].
\end{align*}
For any $t\in[0,T]$ and $x_0\neq x_1$, we have:
\begin{align*}
 & |q(t, x_0) - q(t, x_1)|\\ \leq& \int_0^\infty |\rho(t-\gamma s, x_0+s) - \rho(t-\gamma s, x_1+s)| w(s) \,ds \\
  \leq& \int_0^{t/\gamma} | \rho(t-\gamma s, x_0+s) - \rho(t-\gamma s, x_1+s) | w(s) \,ds + K_0|x_0-x_1|.
\end{align*}
The equation $\rho=\frac{z}{1+\gamma v(q)}$ gives
\begin{align*}
  & | \rho(t-\gamma s, x_0+s) - \rho(t-\gamma s, x_1+s) | \\
  \leq& | z(t-\gamma s, x_0+s) - z(t-\gamma s, x_1+s) | \\&+ \gamma \norm{v'}_\infty | q(t-\gamma s, x_0+s) - q(t-\gamma s, x_1+s) | \\
  \leq& K_1|x_0-x_1| + \gamma \norm{v'}_\infty K(t-\gamma s, |x_0-x_1|) |x_0-x_1|.
\end{align*}
Therefore we have
\begin{align*}
  K(t,r) &\leq K_0 + K_1 + \gamma \norm{v'}_\infty \int_0^{t/\gamma} K(t-\gamma s, r) w(s)\,ds \\
  &= K_0 + K_1 + \norm{v'}_\infty \int_0^t K(\tilde{t}, r) w((t-\tilde{t})/\gamma) \,d\tilde{t},
\end{align*}
for any $t\in[0,T]$ and $r>0$.

Using Gr\"onwall's inequality, we deduce that there exists a constant $K_2>0$ only depending on $K_0,K_1,\gamma,v,w$ such that $K(t,r) \leq K_2$ for any $t\in[0,T]$ and $r>0$, which gives the Lipschitz bound $|\partial_x q(t,x)|\leq K_2$ for $(t,x)\in[0,T]\times\mathbb{R}$.

\vspace{4pt}\noindent
\textbf{Step 6 (Continuation).}
We iteratively construct the solution on time intervals $[t_0,t_1]$, $[t_1,t_2]$, $[t_2,t_3]$, $\cdots$ from $t_0=0$.
At time $t_k$ ($k=0,1,2,\cdots$), the past-time data is given by $\rho(t,x)$ for $(t,x)\in[0,t_k]\times\mathbb{R}$ and $\rho_{-}(t,x)$ for $(t,x)\in(-\infty,0]\times\mathbb{R}$.
Thanks to the $\mathbf{L}^\infty$ and Lipschitz bounds obtained in Step 2, Step 4, and Step 5, the constructed solution on the time interval $[0,t_k]$ satisfies
\[ \rho_{\mathrm{min}} \leq \rho(t,x) \leq \rho_{\mathrm{max}}, \quad g(\rho_{\mathrm{min}}) \leq z(t,x) \leq g(\rho_{\mathrm{max}}), \]
and
\[ \rho \text{ is Lipschitz with } \norm{\rho}_{\mathrm{Lip}} \leq C(\rho_-,\gamma,v,w) Z(t_k), \]
where $Z(t)$ is the solution of the ODE \eqref{eq:lip_ode}.
The above estimates guarantee that $t_k\to\infty$ as $k\to\infty$ and the solution can be extended to the whole domain $(t,x)\in[0,\infty)\times\mathbb{R}$.

\subsection{Discussion}\label{sub:discussion1}

We now make some remarks on the model assumptions and the proof of Theorem~\ref{thm:wellposedness_lipschitz}.

\begin{remark}\label{rem:nonlocal_kernel}
The Assumption 2 requires that the weight kernel $w=w(s)$ has an exponential decay. Such an assumption was also used in \cite{colombo2021local} to establish the nonlocal-to-local limit of the nonlocal-in-space model \eqref{eq:model_nonlocal_space_1}-\eqref{eq:model_nonlocal_space_2}.
In Theorem~\ref{thm:wellposedness_lipschitz}, it is assumed that $\gamma \norm{v'}_\infty w(0) \leq \beta$, which together with \eqref{eq:kernel_decay} gives
\begin{align}\label{eq:kernel_decay_a}
  w'(s) \leq -\gamma\norm{v'}_\infty w(0)w(s).
\end{align}
It is worth noting that if $w=w(s)$ satisfies the condition \eqref{eq:kernel_decay_a}, the rescaled kernel $w_\ep(s)=w(s/\ep)/\ep$ also satisfies the condition with the same parameters $\gamma$ and $\norm{v'}_\infty$.
For the exponential kernel $w=w_\ep(s)$ defined in \eqref{eq:exp_kernel}, the Assumption 2 is satisfied for all $\ep>0$ whenever $\gamma {\norm{v'}_\infty} <1$, which is consistent with the sub-characteristic condition under the relaxation system formulation \eqref{model_relax_1}-\eqref{model_relax_2}.
\end{remark}

\begin{remark}\label{rem:initial_condition}
Let us define the function space \smallskip
\begin{equation}\label{eq:Lip_IC_Def_VertExt}
\begin{split}
&\widetilde{\mathcal{X}}_{L} \doteq \\&\left\{ \rho_0 \in \mathbf{L}^{\infty}(\mathbb{R}) \, : \,
\begin{gathered}
    \inf_{ x \in \mathbb{R}}\rho_{0}(x) > 0, \quad \sup_{x \in \mathbb{R}}\rho_{0}(x) < 1, \\
    \inf_{x\in\mathbb{R}}\rho_0(x) (1 + \gamma v(q_0(x))) > 0, \quad \sup_{x\in\mathbb{R}}\rho_0(x) (1 + \gamma v(q_0(x))) < 1, \\
    |\partial_x \rho_0(x)| \leq L, \quad \sup_{x\in\mathbb{R}} | \partial_x(\rho_0(x)(1+\gamma v(q_0(x)))) | \leq L
\end{gathered}
\right\},
\end{split}
\end{equation}
where the velocity $q_0$ is written as $q_0(x)=\int_0^\infty \rho_0(x+s) w(s) \,ds$.
Then for $\rho_-$ defined via the extension \eqref{eq:extension_vertical} for a given function $\rho_0$,
\begin{equation}
\rho_- \in \mathcal{X}_{\mathrm{Lip},L} \quad \Leftrightarrow \quad
\rho_0 \in \widetilde{\mathcal{X}}_{L}.
\end{equation}
By the form of $q$ we can see that even if $0\leq\rho_0(x)\leq1$ for all $x\in\mathbb{R}$ without an additional condition that the constraint $0\leq \rho_0(x)(1+\gamma v(q_0(x)))\leq1$ can be violated at some point $x\in\mathbb{R}$ where $\rho_0(x)=1$ and $q_0(x)<1$. A sufficient condition on $\rho_0$ alone to ensure $\rho_0 \in \widetilde{\mathcal{X}}_{L}$ is
\[ 0 < \inf_{x\in\mathbb{R}}\rho_0(x), \quad \rho_0(x) \leq \frac{1}{1+\gamma v_{\mathrm{max}}}, \quad |\partial_x \rho_0(x)| \leq \frac{L}{1+\gamma (v_{\max} + \norm{v'}_{\infty}) }. \]
In this case, the lower and upper bounds for the solutions given in Theorem~\ref{thm:wellposedness_lipschitz} become
\[ \inf_{x\in\mathbb{R}}\rho_0(x) \leq \rho(t,x) \leq (1+\gamma v_{\mathrm{max}}) \sup_{x\in\mathbb{R}}\rho_0(x), \quad (t,x)\in(0,\infty)\times\mathbb{R}. \]
These sufficient conditions and solution bounds may not be the best possible results, we will leave possible improvements for the future research.
\end{remark}

\section{Existence, uniqueness and stability of weak solutions}\label{sec:continuous_dependence}

For the remainder of the paper we concern ourselves with a class of past-time data extended vertically from given initial data.
That is, we assume that the past-time data $\rho_{-}\in\mathbf{L}^\infty ((-\infty,0]\times\mathbb{R})$ satisfies \eqref{eq:extension_vertical}
for a given $\rho_0 \in \mathcal{X}$, where $\mathcal{X}$ denotes the class
\begin{equation}\label{eq:Vert_IC_Def}
  \mathcal{X} \doteq \left\{ \rho_0\in\mathbf{L}^\infty(\mathbb{R}) \, : \,
  \begin{gathered}
    0 \leq \rho_0(x) \leq 1, \\
    0 \leq \rho_0(x) (1 + \gamma v(q_0(x))) \leq 1, \\
    \mathrm{TV}(\rho_0) < \infty
  \end{gathered}
  \right\}.
\end{equation}
With this assumption we establish the $\mathbf{L}^1$-stability of Lipschitz solutions, from which Theorem \ref{thm:WeakSolnExistence} 
follows.

\begin{proposition}[$\mathbf{L}^1$-stability of Lipschitz solutions]\label{prop:L1Stab:1}
Under Assumption 1, Assumption 2, and \eqref{eq:GammaSmallAssump}, assume that two functions $\rho^i_0 \in \mathcal{X}_{L_i}$ for $i = 0,1$ (that is, their Lipschitz constants are possibly different).
  Let $\rho^i(t,x)$ for $i = 0,1$ be Lipschitz solutions  to \eqref{eq:model_nonlocal_1}-\eqref{eq:model_nonlocal_3} with initial conditions $\rho^i(0,x) = \rho_0^i(x)$ respectively.

  Then for any $T>0$ there exists a positive constant  $\bar{C} = \bar{C}(v,w,T,\mathrm{TV}(\rho_0^0),\break \mathrm{TV}(\rho_0^1))$ such that
  \begin{equation}\label{eq:L1LocEstimate}
    \begin{split}
      \int_0^T  \int_{\mathbb{R}} \Big( |\rho^1(t,x)-\rho^0(t,x)| +  |q^1(t,x)-q^0(t,x)| \Big) dx dt \leq \bar{C} \int_{\mathbb{R}} |\rho_0^1(x)-\rho^0_0(x)| dx\,.
    \end{split}
  \end{equation}
\end{proposition}

\begin{proof}
  Let $\theta \in [0,1]$ be a parameter; for each value of $\theta$, define $\rho^{\theta}$ to be a Lipschitz solution to \eqref{eq:model_nonlocal_1}-\eqref{eq:model_nonlocal_3} satisfying $0 \leq \rho^{\theta} \leq 1$ with initial data $\rho^{\theta}(0,x) := \theta \rho^0(0,x)+(1-\theta) \rho^1(0,x)$. At least one such solution exists by Theorem \ref{thm:wellposedness_lipschitz} and Remark \ref{rem:initial_condition}.
  Define the first order perturbations for $(t,x) \in [0,\infty) \times \mathbb{R}$:
  \[
  P^\theta(t,x)  ~\doteq~ \lim_{h\to 0} \frac{\rho^{\theta+h}(t,x)-\rho^{\theta}(t,x)}{h},
  \qquad
  Q^\theta(t,x)  ~\doteq~ \lim_{h\to 0} \frac{q^{\theta+h}(t,x)-q^{\theta}(t,x)}{h}.
  \]
  Recalling the definition of the quantity $z^\theta = \rho^\theta(1+\gamma v(q^\theta))$ in \eqref{eq:def_z}, define its first order perturbation as
  \[
  \zeta^\theta(t,x)  ~\doteq~ \lim_{h\to 0} \frac{z^{\theta+h}(t,x)-z^{\theta}(t,x)}{h}.
  \]
  Then
  \begin{equation}\label{eq:FormulaForZeta}
    \zeta^\theta = P^\theta (1+\gamma v(q^\theta)) + \gamma \rho^\theta v'(q^\theta) Q^\theta,
  \end{equation}
  and $\zeta^{\theta}$ satisfies the linearized equation
  \begin{equation}\label{eq:LinearizedEqnForZeta}
    \partial_t \zeta^\theta + \partial_{y}[V(q^\theta) \zeta^\theta + z^\theta V'(q^\theta) Q^\theta] = 0,
  \end{equation}
  where $V(q) \doteq \frac{v(q)}{1+\gamma v(q)}$.

  From \eqref{eq:model_nonlocal_3} the integral defining $Q^{\theta}$ can be written as
  \begin{equation}\label{eq:Rep:GenKernel}
    \begin{split}
      Q^{\theta}(t,x) &= \int_{0}^{\infty} P^{\theta}(0, x+ t/\gamma + s) w(s+t/\gamma) ds  + \int_0^{t/\gamma} P^{\theta}(t-\gamma s,x+s) w(s) \; ds\,.
    \end{split}
  \end{equation}
  We 
  also use a consequence of the condition \eqref{eq:kernel_decay} on $w$:
  \begin{equation}\label{eq:L1Stab:KernelDecay}
    w(s_1) \leq w(s_0) e^{-\beta(s_1-s_0)} \quad \text{ for } 0 \leq s_0 \leq s_1 < \infty.
  \end{equation}
    Third, we note a variant of Gr\"onwall's inequality
    \begin{equation}\label{eq:Gronwall}
        U'(t) \leq u(t) + C U(t)\,, U(0) = 0 \quad \Rightarrow \quad U(t) \leq \int_0^t e^{C(t-s)} u(s) ds\,.
    \end{equation}
  \textbf{Step 1.} We show that along any finite characteristic segment, the perturbed quantity $z^{\theta}$ has bounded total variation.
  To be precise, define
  \begin{equation*}
      G(x,t) := \int_{\frac{t}{3\gamma}}^{\frac{t}{\gamma}} \left| \partial_y z^{\theta} \left( t- \gamma \xi, x + \xi \right) \right| d \xi.
  \end{equation*}
  We claim that there exists $C$ depending only on $v$ and $w$ such that
  \begin{equation}\label{eq:L1Est:BVBound}
    \sup_{\substack{x \in \mathbb{R} \\ t \in [0,T] } } G(x,t) \leq M_T := \mathrm{TV}(\rho^{\theta}_0) \cdot C \mathrm{e}^{C T}
  \end{equation}
  and
  \begin{equation}\label{eq:L1Est:BVBound2}
    \sup_{\substack{ t \in [0,T] } } \int_{\mathbb{R}} G(x,t) \, dx \leq M_T.
  \end{equation}
    We will prove only \eqref{eq:L1Est:BVBound}; \eqref{eq:L1Est:BVBound2} is obtained using the same procedure. From
  \begin{equation*}
    \partial_y ( \partial_t z^{\theta}  ) = - \partial_y ( V(q^{\theta}) \partial_y z^{\theta} ) - \partial_y ( V'(q^{\theta}) z^{\theta }\partial_y q^{\theta}  )\,,
  \end{equation*}
  we have \smallskip
  \begin{equation*}
    \begin{split}
      \frac{d}{dt} [G(x,t)] &= \frac{d}{dt} \left[  \int_{\frac{t}{3\gamma}}^{\frac{t}{\gamma}} \left| \partial_y z^{\theta} \left( t- \gamma \xi, x + \xi \right) \right| d \xi \right] \\
      &= \frac{1}{\gamma}  |\partial_y z^{\theta}(0,x + t/\gamma)| - \frac{1}{3 \gamma}  |\partial_y z^{\theta}(2t/3,x + t/3\gamma)|
      \\
      &\quad + \int_{\frac{t}{3 \gamma}}^{\frac{t}{\gamma}}  \partial_{t} \big[ |\partial_y z^{\theta}(t-\gamma \xi,x+\xi)| \big] d \xi \\
      &=  \Big( \frac{1}{\gamma} - V(q^{\theta}) \Big)  |\partial_y z^{\theta}| (0,x+t/\gamma) + \Big( V(q^{\theta}) - \frac{1}{3\gamma} \Big)  |\partial_y z^{\theta}| (2t/3,x + t/3\gamma)  \\
      &\quad - \int_{\frac{t}{3\gamma}}^{\frac{t}{\gamma}}  \big( \text{sgn} (\partial_y z^{\theta})  \partial_{y} [z^{\theta} V'(q^{\theta}) \partial_y q^{\theta}] \big) (t-\gamma \xi, x+\xi) d \xi.
    \end{split}
  \end{equation*}
  Since $\gamma \| V \|_{\infty} < \frac{1}{3}$ the second term in the integral can be dropped in the estimate, and so
  \begin{equation}\label{eq:L1EstPf:BVBound:MainIntegral}
    \begin{split}
    \frac{d}{dt} [G(x,t)]
    &\leq \frac{C(v)}{\gamma} |\partial_y z^{\theta}(0,x+t/\gamma)| \\
    &\quad + \int_{\frac{t}{3\gamma}}^{\frac{t}{\gamma}}  \big| \partial_{y} [z^{\theta} V'(q^{\theta}) \partial_y q^{\theta}] \big| (t-\gamma \xi, x+\xi) d \xi dx.
    \end{split}
  \end{equation}
We need to estimate the last integral on the right-hand side.
We use \eqref{eq:model_nonlocal_3} to obtain the identities
\begin{equation*}
  \begin{split}
  \partial_y [q^{\theta}(t-\gamma \xi, x + \xi) ] &= \int_0^{\infty} \partial_x \rho^{\theta}(0,x+t/\gamma+s) w(s+t/\gamma - \xi) d s \\
    &\quad + \int_{\xi}^{t/\gamma} \partial_y \rho^{\theta}(t-\gamma s, x+s ) w(s-\xi) ds, \\
  \partial_{yy} [q^{\theta}(t-\gamma \xi, x + \xi)] &= -\int_0^{\infty} \partial_x \rho^{\theta}(0,x+t/\gamma+s) w'(s+t/\gamma - \xi) d s \\
    &\quad - w(0) \partial_y \rho^{\theta}(t-\gamma \xi, x+ \xi) \\-& \int_{\xi}^{t/\gamma} \partial_y \rho^{\theta}(t-\gamma s, x+s ) w'(s-\xi) ds,
    \end{split}
\end{equation*}
from which it follows that
\begin{equation}\label{eq:L1EstPf:BVBound:Q1}
  \begin{split}
    \int_{\frac{t}{3 \gamma}}^{\frac{t}{\gamma}} |\partial_y [q^{\theta}(t-\gamma \xi, x + \xi)]| d \xi
    &\leq  \mathrm{TV}(\rho^{\theta}_0) +
    \int_{\frac{t}{3 \gamma}}^{\frac{t}{\gamma}} |\partial_y \rho^{\theta}(t-\gamma \xi,x + \xi)| d \xi, \\
    \int_{\frac{t}{3 \gamma}}^{\frac{t}{\gamma}} |\partial_{yy} [ q^{\theta}(t-\gamma \xi,x+\xi) ] | d \xi
    &\leq  w(0) \mathrm{TV}(\rho^{\theta}_0) +  2 w(0)
    \int_{\frac{t}{3 \gamma}}^{\frac{t}{\gamma}} |\partial_y \rho^{\theta}(t-\gamma \xi,x + \xi)| d \xi.
  \end{split}
\end{equation}
In the same way, we can obtain the bound
\begin{equation}\label{eq:L1EstPf:Q2}
  \sup_{\substack{ \xi \in (\frac{t}{3 \gamma},\frac{t}{ \gamma}) }} \norm{ \partial_y [q^{\theta}(t-\gamma \xi, \cdot + \xi)] }_{\infty} = \sup_{\substack{ \xi \in (\frac{t}{3 \gamma},\frac{t}{ \gamma})  }} \norm{  \partial_\xi [q^{\theta}(t-\gamma \xi, \cdot+ \xi)]  }_{\infty} \leq 3 w(0).
\end{equation}
The estimates \eqref{eq:L1EstPf:BVBound:Q1} and \eqref{eq:L1EstPf:Q2} are applied to majorize the integral on the last line of \eqref{eq:L1EstPf:BVBound:MainIntegral} by \smallskip
\begin{equation}\label{eq:L1Est:BVBound:MajorBound}
  \begin{split}
    &\|V''\|_{\infty}   \|z^{\theta}\|_{\infty} \sup_{\substack{ \xi \in (\frac{t}{3 \gamma},\frac{t}{ \gamma}) }} \norm{  \partial_y q^{\theta}(t-\gamma \xi, \cdot+ \xi)  }_{\infty}
    \cdot \int_{\frac{t}{3 \gamma}}^{\frac{t}{\gamma}} |\partial_y q^{\theta}(t-\gamma \xi,x+\xi)| d \xi \\
    &\qquad + \|V'\|_{\infty} \sup_{\substack{ \xi \in (\frac{t}{3 \gamma},\frac{t}{ \gamma}) }} \norm{  \partial_y q^{\theta}(t-\gamma \xi, \cdot+ \xi)  }_{\infty} \cdot
    \int_{\frac{t}{3 \gamma}}^{\frac{t}{\gamma}} |\partial_y z^{\theta}(t-\gamma \xi,x+\xi)| d \xi  \\
    &\qquad +  \|V'\|_{\infty}
    \| z^{\theta} \|_{\infty} \int_{\frac{t}{3 \gamma}}^{\frac{t}{\gamma}} |\partial_{yy} q^{\theta}(t-\gamma \xi, x+\xi)| d \xi \\
    &\leq C(v,w) \left[ \mathrm{TV}(\rho^{\theta}_0) +
    \int_{\frac{t}{3 \gamma}}^{\frac{t}{\gamma}} |\partial_y \rho^{\theta}(t-\gamma \xi,x + \xi)|  d \xi +  \int_{\frac{t}{3 \gamma}}^{\frac{t}{\gamma}} |\partial_y  z^{\theta}(t-\gamma \xi,x + \xi)|  d \xi  \right].
  \end{split}
\end{equation}
Now, since $z = \rho(1+\gamma v(q))$ we have that $(1-\gamma v_{\mathrm{max}})|\partial_y \rho| \leq |\partial_y z| + \gamma \| v' \|_{\infty} |\partial_y q|$, and so along with \eqref{eq:L1EstPf:BVBound:Q1}
\begin{equation*}
  \begin{split}
    &\int_{\frac{t}{3 \gamma}}^{\frac{t}{\gamma}} |\partial_y \rho^{\theta}(t-\gamma \xi,x+\xi)| d \xi\\
    &\leq \int_{\frac{t}{3 \gamma}}^{\frac{t}{\gamma}} |\partial_y z^{\theta}(t-\gamma \xi,x+\xi)| d \xi  + \frac{\gamma \| v' \|_{\infty} }{ 1-\gamma v_{\mathrm{max}} } \int_{\frac{t}{3 \gamma}}^{\frac{t}{\gamma}} |\partial_y q^{\theta}(t-\gamma \xi,x+\xi)| d \xi  \\
    &\leq \mathrm{TV}(\rho^{\theta}_0) + \int_{\frac{t}{3 \gamma}}^{\frac{t}{\gamma}} |\partial_y z^{\theta}(t-\gamma \xi,x+\xi)| d \xi  + \frac{\gamma \| v' \|_{\infty} }{ 1-\gamma v_{\mathrm{max}} } \int_{\frac{t}{3 \gamma}}^{\frac{t}{\gamma}} |\partial_y \rho^{\theta}(t-\gamma \xi,x+\xi)| d \xi.
  \end{split}
\end{equation*}
Since $\frac{\gamma \| v' \|_{\infty} }{ 1-\gamma v_{\mathrm{max}} } <1/3$ by assumption we can absorb the last term into the left-hand side of the estimate to get
\begin{equation}\label{eq:L1EstPf:BVBound:AbsorbBd}
  \int_{\frac{t}{3\gamma}}^{\frac{t}{\gamma}} |\partial_y \rho^{\theta}(t-\gamma \xi,x+\xi)| d \xi  \leq \frac{3}{2} \left( \mathrm{TV}(\rho^{\theta}_0) + \int_{\frac{t}{3\gamma}}^{\frac{t}{\gamma}} |\partial_y z^{\theta}(t-\gamma \xi,x+\xi)| d \xi \right).
\end{equation}
Inserting \eqref{eq:L1EstPf:BVBound:AbsorbBd} into \eqref{eq:L1Est:BVBound:MajorBound}, the estimate for the total variation of $z^{\theta}$ from \eqref{eq:L1EstPf:BVBound:MainIntegral} is now
\begin{equation*}
  \frac{d}{dt} \left[  G(x,t) \right]  \leq \frac{C(v)}{\gamma} |\partial_y z^{\theta}(0,x+t/\gamma)| + C(v,w) \mathrm{TV}(\rho^{\theta}_0) +  C(v,w) G(x,t).
\end{equation*}
Then by \eqref{eq:Gronwall}
\begin{equation*}
    \begin{split}
        G(x,t) &\leq \bar{C} e^{\bar{C}t} \int_0^t \left( \mathrm{TV}(\rho^{\theta}_0) + \frac{1}{\gamma} |\partial_y z^{\theta}(0,x+s/\gamma)| \right) ds \\
        &\leq \bar{C} e^{\bar{C}t} \mathrm{TV}(\rho^{\theta}_0) + \bar{C} e^{\bar{C}t} \int_{x}^{x+t/\gamma} |\partial_y z^{\theta}(0,\xi)| d \xi \leq \bar{C} e^{\bar{C}t} \mathrm{TV}(\rho^{\theta}_0),
    \end{split}
\end{equation*}
where $\bar{C}$ depends only on $v$ and $w$.
The bound \eqref{eq:L1Est:BVBound} follows.

\vspace{4pt}\noindent
  \textbf{Step 2.} We prove the main result. The method is similar to Step 1.
  Define $E: [0,\infty) \to \mathbb{R}$ by
  \begin{equation*}
    \begin{split}
      E(t) := \int_{\frac{t}{3 \gamma}}^{\frac{t}{\gamma}} \int_{\mathbb{R}} |\zeta^{\theta}(t-\gamma \xi,x + \xi)| d x d \xi = \frac{1}{\gamma} \int_0^{\frac{2t}{3}} \int_{\mathbb{R}} |\zeta^{\theta}(\tau,x)| d x d \tau.
    \end{split}
  \end{equation*}
  We use the linearized equation \eqref{eq:LinearizedEqnForZeta} and apply integration by parts to obtain \smallskip
  \begin{equation*}
    \begin{split}
      \frac{d}{dt} E(t)
      &= \frac{1}{\gamma} \int_{\mathbb{R}} |\zeta^{\theta}(0,x + t/\gamma)| dx - \frac{1}{3 \gamma} \int_{\mathbb{R}} |\zeta^{\theta}(2t/3,x + t/3\gamma)| dx
      \\
      &\quad + \int_{\frac{t}{3 \gamma}}^{\frac{t}{\gamma}} \int_{\mathbb{R}}  \partial_{t} \big[ |\zeta^{\theta}(t-\gamma \xi,x+\xi)| \big] dx d \xi \\
      &= \int_{\mathbb{R}} \Bigg[ \Big( \frac{1}{\gamma} - V(q^{\theta}) \Big)  |\zeta^{\theta}| (0,x+t/\gamma) + \Big( V(q^{\theta}) - \frac{1}{3\gamma} \Big)  |\zeta^{\theta}| (2t/3,x + t/3\gamma)  \\
        &\quad - \int_{\frac{t}{3\gamma}}^{\frac{t}{\gamma}}  \big( \text{sgn} (\zeta^{\theta})  \partial_{y} [z^{\theta} V'(q^{\theta}) Q^{\theta}] \big) (t-\gamma \xi, x+\xi) d \xi \Bigg] dx.
    \end{split}
  \end{equation*}
  Since $\gamma \| V \|_{\infty} < \frac{1}{3}$, the second term in the integral can be dropped, and so
  \begin{equation}\label{eq:L1EstPf:SecondIntegral}
    \frac{d}{dt} E(t) \leq C(v) \left( \frac{1}{\gamma} \| \zeta^{\theta}(0,\cdot) \|_{1}
    +\int_{\mathbb{R}} \int_{\frac{t}{3\gamma}}^{\frac{t}{\gamma}}  \big| \partial_{y} [z^{\theta} V'(q^{\theta}) Q^{\theta}] \big| (t-\gamma \xi, x+\xi) d \xi dx \right).
  \end{equation}
  We need to estimate the last integral on the right-hand side.
  We use \eqref{eq:Rep:GenKernel} and \eqref{eq:L1Stab:KernelDecay} to obtain the estimates
  \begin{equation}\label{eq:L1EstPf:Q1}
    \begin{split}
      &\int_{\frac{t}{3 \gamma}}^{\frac{t}{\gamma}} \int_{\mathbb{R}} |Q^{\theta}(t-\gamma \xi, x + \xi)| d \xi
      \leq  \beta^{-1} \| P^{\theta}(0,\cdot) \|_{1} +
      \int_{\frac{t}{3\gamma}}^{\frac{t}{\gamma}}  \int_{\mathbb{R}} |P^{\theta}(t-\gamma s,x + s)| dx ds, \\
      &\sup_{\xi \in (\frac{t}{3 \gamma},\frac{t}{\gamma})} |Q^{\theta}(t-\gamma \xi,x+\xi)| \leq \| P^{\theta}(0,\cdot) \|_{1} + w(0)
      \int_{\frac{t}{3 \gamma}}^{\frac{t}{\gamma}}  |P^{\theta}(t-\gamma s,x+s)| d s, \\
      &\quad\int_{\frac{t}{3\gamma}}^{\frac{t}{\gamma}} \int_{\mathbb{R}} |\partial_{y} [ Q^{\theta}(t-\gamma \xi,x+\xi) ] | d x d \xi
      \\&= \int_{\frac{t}{3\gamma}}^{\frac{t}{\gamma}} \int_{\mathbb{R}} | \partial_{\xi} [ Q^{\theta}(t-\gamma \xi,x+\xi) ] | d x d \xi \\
      &=  \int_{\frac{t}{3\gamma}}^{\frac{t}{\gamma}} \int_{\mathbb{R}} \bigg| \int_0^{\infty} P^{\theta}(0,x+t/\gamma+s) (-w'(s-\xi+t/\gamma)) ds \\
      &\qquad  + P^{\theta}(t-\gamma \xi, x+\xi) w(0) - \int_{\xi}^{\frac{t}{\gamma}} P^{\theta}(t-\gamma s,x+s) w'(s-\xi) ds  \bigg| d x d \xi \\
      &\leq  \| P^{\theta}(0,\cdot) \|_{1} + 2 w(0)
      \int_{\frac{t}{3\gamma}}^{\frac{t}{\gamma}}  \int_{\mathbb{R}} |P^{\theta}(t-\gamma s,x + s)| dx ds.
    \end{split}
  \end{equation}
  Then \eqref{eq:L1EstPf:Q1}, \eqref{eq:L1EstPf:Q2}, \eqref{eq:L1Est:BVBound} and \eqref{eq:L1Est:BVBound2} are applied to majorize the last integral in \eqref{eq:L1EstPf:SecondIntegral} by
  \begin{equation}\label{eq:L1Est:MajorBound}
    \begin{split}
      &\|V''\|_{\infty}   \|z^{\theta}\|_{\infty} \sup_{\substack{ \xi \in (t/3\gamma,t/\gamma) }} \norm{  \partial_y q^{\theta}(t-\gamma \xi, \cdot+ \xi)  }_{\infty}
      \cdot \int_{\frac{t}{3\gamma}}^{\frac{t}{\gamma}} \int_{\mathbb{R}}  |Q^{\theta}(t-\gamma \xi,x+\xi)| d x d \xi \\
      &\qquad + \|V'\|_{\infty} \|z^{\theta}\|_{\infty}
      \int_{\frac{t}{3\gamma}}^{\frac{t}{\gamma}} \int_{\mathbb{R}}  |\partial_y [Q^{\theta}(t-\gamma \xi,x+\xi)]| d x d \xi  \\
      &\qquad \quad +  \|V'\|_{\infty} \int_{\mathbb{R}} \left(
      \sup_{\xi \in (\frac{t}{3 \gamma},\frac{t}{\gamma})} |Q^{\theta}(t-\gamma \xi,x+\xi) | \right) \int_{\frac{t}{3\gamma}}^{\frac{t}{\gamma}} |\partial_y z^{\theta}(t-\gamma \xi, x+\xi)| d \xi d x \\
       \end{split}
  \end{equation}
      \begin{equation*}
    \begin{split}
      &\leq C(v,w,M_T) \left[ \| P^{\theta}(0,\cdot) \|_{1} +
      \int_{\frac{t}{3\gamma}}^{\frac{t}{\gamma}}  \int_{\mathbb{R}} |P^{\theta}(t-\gamma s,x + s)| dx ds\right].
    \end{split}
  \end{equation*}
  Now, as a consequence of \eqref{eq:FormulaForZeta} we have
  \begin{equation}\label{eq:rhoAndqAndz}
      (1 - \gamma v_{\mathrm{max}}) |P^{\theta}| \leq |\zeta^{\theta}| + \gamma \| v' \|_{\infty} |Q^{\theta}|,
    \end{equation}
  so with \eqref{eq:L1EstPf:Q1} and the conditions \eqref{eq:GammaSmallAssump} on $\gamma$ and $v$
  \begin{equation*}
    \begin{split}
       & \int_{\frac{t}{3 \gamma}}^{\frac{t}{\gamma}} \int_{\mathbb{R}}  |P^{\theta}(t-\gamma \xi, x + \xi)| dx  d \xi\\
      &\leq \frac{1}{1-\gamma v_{\mathrm{max}}} \int_{\frac{t}{3 \gamma}}^{\frac{t}{\gamma}} \int_{\mathbb{R}} |\zeta^{\theta}(t-\gamma \xi, x + \xi)| dx  d \xi \\
          &\qquad +  \frac{\gamma \| v' \|_{\infty}}{1-\gamma v_{\mathrm{max}}}  \int_{\frac{t}{3 \gamma}}^{\frac{t}{\gamma}} \int_{\mathbb{R}}  |Q^{\theta}(t-\gamma \xi, x + \xi)| dx  d \xi \\
      &\leq 2 E(t) + \frac{1}{3 \beta} \| P^{\theta}(0,\cdot) \|_{1} + \frac{1}{3}
      \int_{\frac{t}{3\gamma}}^{\frac{t}{\gamma}}  \int_{\mathbb{R}} |P^{\theta}(t-\gamma s,x + s)| dx ds.
    \end{split}
  \end{equation*}
  Therefore we can absorb the last term into the left-hand side of the estimate to get
  \begin{equation}\label{eq:L1EstPf:AbsorbBd}
    \int_{\frac{t}{3 \gamma}}^{\frac{t}{\gamma}} \int_{\mathbb{R}}  |P^{\theta}(t-\gamma \xi, x + \xi)| dx  d \xi \leq C(\beta) \left( E(t) + \| P^{\theta}(0,\cdot) \|_{1} \right).
  \end{equation}
  Inserting \eqref{eq:L1EstPf:AbsorbBd} into \eqref{eq:L1Est:MajorBound}, the estimate for the derivative of $E(t)$ from \eqref{eq:L1EstPf:SecondIntegral} is now
  \begin{equation*}
    \frac{d}{d t} E(t)  \leq  C(v,w,T) \left(  \gamma^{-1} \| P^{\theta}(0,\cdot) \|_{1} + E(t) \right);
  \end{equation*}
  the bound $\| \zeta^{\theta}(0,\cdot) \|_{1} \leq C(\gamma,v) \| P^{\theta}(0,\cdot) \|_{1}$ is easily seen from \eqref{eq:FormulaForZeta} and \eqref{eq:Rep:GenKernel}.
  Applying Gr\"onwall's inequality and changing coordinates, we obtain
  \begin{equation}\label{eq:L1Est:FinalZetaEst}
    \begin{split}
      \int_0^T \int_{\mathbb{R}} |\zeta^{\theta}(t,x)| dx dt \leq \bar{C}(v,w,T)  \| P^{\theta}(0,\cdot) \|_{1}.
    \end{split}
  \end{equation}

  Now, by \eqref{eq:L1EstPf:Q1}
  \begin{equation*}
    \begin{split}
     & \int_0^T \int_{\mathbb{R}} |Q^{\theta}(t,x)| dx d t \\&\leq \frac{\gamma}{\beta}  \| P^{\theta}(0,\cdot) \|_{1} + \int_0^T \int_{\mathbb{R}} |P^{\theta}(t,x)| dx d t \\
      &\leq \frac{\gamma}{\beta}  \| P^{\theta}(0,\cdot) \|_{1} + C \int_0^T \int_{\mathbb{R}} |\zeta^{\theta}(t,x)| dx d t + \frac{\gamma \| v' \|_{\infty} }{1- \gamma v_{\mathrm{max}}} \int_0^T \int_{\mathbb{R}} |Q^{\theta}(t,x)| dx d t,
    \end{split}
  \end{equation*}
  where we used that $P^{\theta}$ satisfies \eqref{eq:rhoAndqAndz}.
Since $\frac{\gamma \| v' \|_{\infty} }{1- \gamma v_{\mathrm{max}}} < \frac{1}{3}$ we can absorb the $Q^{\theta}$ term and then apply \eqref{eq:L1Est:FinalZetaEst} to get
\begin{equation*}
  \int_0^T \int_{\mathbb{R}} |Q^{\theta}(t,x)| dx dt \leq \bar{C}(v,w,T)  \| P^{\theta}(0,\cdot) \|_{1}.
\end{equation*}
Therefore the estimates for $\zeta^{\theta}$ and $Q^{\theta}$ combine using \eqref{eq:rhoAndqAndz} to give us the estimate for $P^{\theta}$:
\begin{equation*}
  \int_0^T \int_{\mathbb{R}} |P^{\theta}(t,x)| dx dt \leq \bar{C}(v,w,T)  \| P^{\theta}(0,\cdot) \|_{1}.
\end{equation*}
  To conclude the proof, we use the
  above two inequalities to get:
  \begin{equation*}
  \begin{split}
    \int_0^T \int_{\mathbb{R}} &\big( |\rho^1(t,x)-\rho^0(t,x)| + |q^1(t,x)-q^0(t,x)| \big) dx dt \\
    &\leq \int_0^1 \int_0^T \int_{\mathbb{R}}  \Big( |P^{\theta}(t,x)| + |Q^{\theta}(t,x)| \Big) dx dt  d\theta \\
    &\leq \int_0^1  \bar{C}(T)  \| P^{\theta}(0,\cdot) \|_{1} d \theta
    \leq  \bar{C}(T) \int_{\mathbb{R}} |\rho^1(0,x)-\rho^0(0,x)| dx.
  \end{split}
  \end{equation*}
\vspace{-10pt}\end{proof}

\begin{proof}[Proof of Theorem \ref{thm:WeakSolnExistence}]
  Let $\rho_0 \in \mathcal{X}$, and let $\rho_0^n$, $n \in \mathbb{N}$, be a sequence of mollified functions in $\widetilde{\mathcal{X}}_{L_n}$
  (possibly with $L_n \to \infty$) that converge to $\rho_0$ in $\mathbf{L}^1_{\mathrm{loc}}(\mathbb{R})$. By virtue of \eqref{eq:L1LocEstimate} the corresponding solutions $\rho^n \in \mathcal{D}_{L_n,T,\rho_{\mathrm{min}},\rho_{\mathrm{max}}}$ to \eqref{eq:model_nonlocal_1}-\eqref{eq:model_nonlocal_3} with initial condition $\rho^n(0,x) = \rho_0^n(x)$ are Cauchy, and hence converge, in $\mathbf{L}^1_{\mathrm{loc}}([0,T] \times \mathbb{R})$ to a function $\rho$. Thus $\rho$ satisfies \eqref{eq:WeakSolnDef}, and so is a weak solution.
  Furthermore, we note that the weak solutions constructed in this way inherit the same stability property \eqref{eq:L1LocEstimate}, since the bounding constant in that inequality does not depend on the Lipschitz constant of the solutions, and so uniqueness follows.
  To complete the proof, 
  given that
  $\rho^n$ is a bounded sequence in $\mathbf{L}^{\infty}([0,T] \times \mathbb{R})$, 
  and  the
  weak-$*$ limits are unique, by noting the sequence $\rho^n$ is obtained with initial conditions that are mollified approximations of $\rho_0$, we can pass through the limits to obtain the bounds \eqref{eq:sol_lower_bound}-\eqref{eq:sol_upper_bound} for the weak solution $\rho$.
\end{proof}

\section{Uniform BV bound and existence of limit solutions}
\label{sec:BV_bound}

Towards the aim of proving the convergence of the solutions
of \eqref{eq:model_nonlocal_1}-\eqref{eq:model_nonlocal_3}
as the weight kernel $w$ converges to a Dirac delta function, we consider only the exponential kernels as defined in \eqref{eq:exp_kernel}:
\begin{align*}
  w(s) = e^{-s}, \qquad
  w_\ep(s) = \ep^{-1} w(s/\ep) = \ep^{-1} e^{-s/\ep},\quad s\in[0,\infty).
\end{align*}
In this case the nonlocal model \eqref{eq:model_nonlocal_1}-\eqref{eq:model_nonlocal_3}
can be reformulated as the relaxation system \eqref{model_relax_1}-\eqref{model_relax_2}, which is recalled here: \begin{align*}
  \partial_t \rho + \partial_x (\rho v(q)) &=0, \\
  \partial_t q - \gamma^{-1}\partial_x q &= (\gamma\ep)^{-1} (\rho - q).
\end{align*}
The characteristic speeds of the system are
\[
\lambda_1 = -\gamma^{-1} < 0, \quad \lambda_2 = v(q) \geq 0.
\]
Taking $\ep\to0$, we expect the solution of \eqref{model_relax_1}-\eqref{model_relax_2} to converge to that of its equilibrium approximation,  which is the LWR model \eqref{eq:local_model}.
The characteristic speed of the limit equation \eqref{eq:local_model} is
\[
\lambda = v(\rho) + \rho v'(\rho).
\]

The condition \eqref{eq:GammaSmallAssump} plus $\rho\geq\rho_{\mathrm{min}}>0$ ensures the strict sub-characteristic condition $\lambda_1<\lambda<\lambda_2$.

\subsection{Uniform BV bound}\label{sub:uniformBVbound}

\begin{proof}[Proof of Theorem \ref{thm:BV_bound}]
Let us first assume $\rho_0\in\mathbf{C}^2_{\mathrm{c}}(\mathbb{R})$.
In this case, $\rho$ and $q$ are Lipschitz continuous and satisfy the reformulated system \eqref{model_relax_1}-\eqref{model_relax_2} pointwise.

Noting that $\rho$ and $1+\gamma v(q)$ stay positive provided $\rho_{\mathrm{min}}>0$, we construct
\begin{align}
  u  = \ln (\rho (1+\gamma v(q)) ), \quad h = -\ln (1+\gamma v(q)).
\end{align}
One can easily verify that $u$ and $h$ are Riemann invariants of the system \eqref{model_relax_1}-\eqref{model_relax_2} corresponding to the system's characteristic speeds $\lambda_2=v(q)$ and $\lambda_1=-\gamma^{-1}$, respectively.
With the new set of variables $(u,h)$, the system \eqref{model_relax_1}-\eqref{model_relax_2} can be diagonalized as
\begin{align}
\partial_tu + v(q(h)) \partial_xu =& \ep^{-1}\Lambda(u,h), \label{eq:system_quasi_monotone_1}\\
\partial_th - \gamma^{-1} \partial_xh =& -\ep^{-1}\Lambda(u,h), \label{eq:system_quasi_monotone_2}
\end{align}
where $q(h) \doteq v^{-1} \left( \gamma^{-1}(e^{-h}-1) \right)$ is an increasing function, and
\begin{align}
  \Lambda(u,h) = v'(q(h)) e^{h} \left( e^{u+h} - q(h) \right).
\end{align}

Note that $u(0,\cdot),h(0,\cdot)\in\mathbf{C}^2_{\mathrm{c}}(\mathbb{R})$ and $u,h$ are Lipschitz continuous. By the method of characteristics we see that $\partial_xu,\partial_xh$ are Lipschitz continuous and compactly supported.
We claim that the system
\eqref{eq:system_quasi_monotone_1}-\eqref{eq:system_quasi_monotone_2}
is total variation diminishing, i.e.,
\begin{equation}\label{eq:TVuh}
\frac{d}{dt} \int_{\mathbb{R}}| \partial_xu| + |\partial_xh| \,dx \leq 0.
\end{equation}
Indeed, differentiating
\eqref{eq:system_quasi_monotone_1}-\eqref{eq:system_quasi_monotone_2} with respect to $x$ gives
\begin{align*}
  \partial_t(\partial_xu) + \partial_x\left( v(q(h)) \partial_xu \right) &= \ep^{-1} (\partial_u\Lambda \cdot \partial_xu + \partial_h\Lambda \cdot \partial_xh), \\
  \partial_t(\partial_xh) + \partial_x\left( -\gamma^{-1} \partial_xh \right) &= -\ep^{-1} (\partial_u\Lambda \cdot \partial_xu + \partial_h\Lambda \cdot \partial_xh),
\end{align*}
from which we obtain that
\begin{align*}
  \frac{d}{dt} \int_{\mathbb{R}}| \partial_xu| + |\partial_xh| \,dx = \int_{\mathbb{R}} \mathrm{sgn}(\partial_xu) \cdot \partial_t(\partial_xu) + \mathrm{sgn}(\partial_xh) \cdot \partial_t(\partial_xh) \,dx = J_1 + J_2,
\end{align*}
where
\begin{align*}
  J_1 &= \int_{\mathbb{R}} -\mathrm{sgn}(\partial_xu) \cdot \partial_x\left( v(q(h)) \partial_xu \right) + \gamma^{-1} \mathrm{sgn}(\partial_xh) \cdot \partial_x(\partial_xh) \,dx \\
  &= \int_{\mathbb{R}} \delta(\partial_xu)v(q(h)) \partial_xu\partial_x^2u - \gamma^{-1} \delta(\partial_xh)\partial_xh\partial_x^2h \,dx \\
  &= 0,
\end{align*}
and
\begin{align*}
  J_2 &= \ep^{-1} \int_{\mathbb{R}} \mathrm{sgn}(\partial_xu) ( \partial_u\Lambda \cdot \partial_xu + \partial_h\Lambda \cdot \partial_xh ) - \mathrm{sgn}(\partial_xh) (\partial_u\Lambda \cdot \partial_xu + \partial_h\Lambda \cdot \partial_xh ) \,dx \\
  &\leq \ep^{-1} \int_{\mathbb{R}} (|\Lambda_u| + \Lambda_u) |\partial_xu| + (|\Lambda_h| - \Lambda_h) |\partial_xh| \,dx.
\end{align*}
A direct calculation gives
\[
\partial_u\Lambda = v'(q(h)) e^{u+2h} \leq 0 \]
and
{\small \begin{align*}
  &\partial_h\Lambda\\ &= e^{h} \left[ \frac{v''(q(h))(1+\gamma v(q(h)))}{\gamma v'(q(h))} (q(h)-e^{u+h}) + v'(q(h)) (2e^{u+h} - q(h)) + v(q(h)) + \frac1\gamma \right] \\
  &\geq e^{h} \left[ \frac{1}{\gamma} - 2\norm{v'}_{\infty} - \frac{(1+\gamma v_{\mathrm{max}}) \norm{v''}_{\infty}}{\gamma \min_{\rho\in[0,1]}|v'(\rho)|} \right] \\
  &\geq0,
\end{align*}}
where the condition \eqref{eq:BV_condition_gamma} and the solution bounds $0 < e^{u+h} = \rho \leq 1 ,~ 0 \leq q(h) \leq 1$ are used.
With $\partial_u\Lambda\leq0$ and $\partial_h\Lambda\geq0$,
the estimate \eqref{eq:TVuh} follows immediately.

Thanks to the estimate \eqref{eq:TVuh}, we now turn to the uniform BV bound on $\rho$. At the initial time $t=0$, we have
\begin{align*}
  \int_{\mathbb{R}} |\partial_x\rho(0,x)| \,dx = \int_{\mathbb{R}} |\partial_xq(0,x)| \,dx = \mathrm{TV}(\rho_0).
\end{align*}
Therefore,
\begin{align*}
  \int_{\mathbb{R}} |\partial_x u(0,x)| + |\partial_x h(0,x)| \,dx &\leq \int_{\mathbb{R}} \frac{1}{\rho(0,x)}|\partial_x \rho(0,x)| + \frac{2\gamma |v'(q(0,x))|}{1+\gamma v(q(0,x))} |\partial_x q(0,x)| \,dx \\
  &\leq \rho_{\mathrm{min}}^{-1} \int_{\mathbb{R}} |\partial_x \rho(0,x)| \,dx + 2\gamma\norm{v'}_{\infty} \int_{\mathbb{R}} |\partial_x q(0,x)| \,dx \\
  &\leq \left( \rho_{\mathrm{min}}^{-1} + 2\gamma\norm{v'}_{\infty} \right) \mathrm{TV}(\rho_0).
\end{align*}
Since the total variation of $(u,h)$ is diminishing, it holds that
\begin{align*}
  \int_{\mathbb{R}} |\partial_x u(t,x)| + |\partial_x h(t,x)| \,dx \leq \left( \rho_{\mathrm{min}}^{-1} + 2\gamma\norm{v'}_{\infty} \right) \mathrm{TV}(\rho_0),
\end{align*}
for any time $t\geq0$.
Noting that $\partial_x\rho = \rho (\partial_xu + \partial_xh)$, we deduce that
\begin{align*}
  \int_{\mathbb{R}} |\partial_x \rho(t,x)| \,dx \leq \int_{\mathbb{R}} |\partial_x u(t,x)| + |\partial_x h(t,x)| \,dx \leq \left( \rho_{\mathrm{min}}^{-1} + 2\gamma\norm{v'}_{\infty} \right) \mathrm{TV}(\rho_0).
\end{align*}
Then, using \eqref{eq:model_nonlocal_3} and \eqref{eq:model_nonlocal_1}, we have
\begin{align*}
  \int_{\mathbb{R}} |\partial_x q(t,x)| \,dx \leq& \left( \rho_{\mathrm{min}}^{-1} + 2\gamma\norm{v'}_{\infty} \right) \mathrm{TV}(\rho_0), \\
  \int_{\mathbb{R}} |\partial_t \rho(t,x)| \,dx \leq& \left( v_{\mathrm{max}} + \norm{v'}_\infty \right) \left( \rho_{\mathrm{min}}^{-1} + 2\gamma\norm{v'}_{\infty} \right) \mathrm{TV}(\rho_0),
\end{align*}
for any time $t\geq0$. Combining the above inequalities, we obtain
\begin{align*}
&  \int_0^T \int_{\mathbb{R}} |\partial_t\rho(t,x)|+|\partial_x\rho(t,x)| \,dxdt\\& \leq (v_{\mathrm{max}} + \norm{v'}_\infty + 1) \left( \rho_{\mathrm{min}}^{-1} + 2\gamma\norm{v'}_{\infty} \right) T \cdot \mathrm{TV}(\rho_0),
\end{align*}
which gives the desired uniform BV bound \eqref{eq:uniform_BV_bound}.

For general initial data $\rho_0 \in \mathcal{X}$, we apply an approximation argument as in Theorem~\ref{thm:WeakSolnExistence} but instead using $\mathbf{C}^2_{\mathrm{c}}(\mathbb{R})$ functions. By passing through the limit we deduce that the BV bound \eqref{eq:uniform_BV_bound} holds also for weak solutions of \eqref{eq:model_nonlocal_1}-\eqref{eq:model_nonlocal_3}.
\end{proof}

\begin{remark}
  A counterexample was given in \cite{colombo2021local} to show that the total variation of solutions to the nonlocal-in-space model \eqref{eq:model_nonlocal_space_1}-\eqref{eq:model_nonlocal_space_2} blow up as $\ep\to0$ if the initial data are not uniformly positive. We leave the same question for \eqref{eq:model_nonlocal_1}-\eqref{eq:model_nonlocal_3} to future works.
\end{remark}

\subsection{Convergence to a weak solution}\label{sub:weak}
Now we are in a position to show the existence of limit solutions that satisfy the limit equation \eqref{eq:local_model} in the weak sense. To pass the limit we need
to establish the following theorem.

\begin{theorem}\label{thm:limit_sol}
Under the same assumptions as in Theorem~\ref{thm:BV_bound}, let $\rho^\ep$ be the unique weak solution of \eqref{eq:model_nonlocal_1}-\eqref{eq:model_nonlocal_3} with parameter $\varepsilon$ and initial condition $\rho^{\varepsilon}(0,x) = \rho_0(x)$. There is a sequence $\ep_n\to0$ and a limit function $\rho^\star\in\mathbf{L}^\infty([0,\infty)\times\mathbb{R})$ such that $\rho^{\ep_n}\to\rho^\star$ in $\mathbf{L}^1_{\mathrm{loc}}([0,\infty)\times\mathbb{R})$. Moreover, $\rho^\star$ is a weak solution of \eqref{eq:local_model}.
\end{theorem}

\begin{proof}
By Theorem~\ref{thm:WeakSolnExistence} and Theorem~\ref{thm:BV_bound}, the family of solutions $\{\rho^\ep\}_{\ep>0}$ is uniformly bounded in $\mathrm{BV}_{\mathrm{loc}}([0,\infty) \times \mathbb{R})$. As a consequence, the family $\{\rho^\ep\}_{\ep>0}$ is precompact in the $\mathbf{L}^1_{\mathrm{loc}}$ norm (see \cite{evans2018measure}).
Then we can select a sequence $\ep_n\to0$ such that $\rho^{\ep_n}\to\rho^\star$ in $\mathbf{L}^1_{\mathrm{loc}}([0,\infty)\times\mathbb{R})$, where the limit function $\rho^\star\in\mathbf{L}^\infty([0,\infty)\times\mathbb{R})$.

Now we claim that
\begin{align}\label{eq:L1estimate_q_minus_rho}
  \int_0^T \int_{\mathbb{R}} |q^\ep(t,x) - \rho^\ep(t,x)| \,dxdt \leq CT\ep \quad \forall T>0,
\end{align}
where the constant $C=C\left(\gamma,v,\rho_{\mathrm{min}}^{-1},\mathrm{TV}(\rho_0) \right)$ is independent of $\ep$.
Indeed, by \eqref{eq:model_nonlocal_3} we can write
\begin{align*}
 & q^\ep(t,x) - \rho^\ep(t,x) \\=& \int_0^{t/\gamma} (\rho^\ep(t-\gamma s,x+s) - \rho^\ep(t,x)) w_\ep(s) \,ds + \int_{t/\gamma}^\infty (\rho_0(x+s) - \rho_0(x)) w_\ep(s) \,ds\\
  &+ (\rho_0(x) - \rho^\ep(t,x)) \int_{t/\gamma}^\infty  w_\ep(s) \,ds,
\end{align*}
where $w_\ep(s)=\ep^{-1}e^{-s/\ep}$.
Integrating the above inequality on $[0,T]\times\mathbb{R}$ and applying Theorem~\ref{thm:BV_bound}, we obtain that
\begin{align*}
  \int_0^T \int_{\mathbb{R}} |q^\ep(t,x) - \rho^\ep(t,x)| \,dxdt \leq J_1 + J_2 + J_3,
\end{align*}
where
\begin{align*}
  J_1 =& \int_0^T \int_0^{t/\gamma} \int_0^s \left(\int_{\mathbb{R}} |(\partial_x-\gamma\partial_t) \rho^\ep(t-\gamma \sigma,x+\sigma)| \,dx \right) w_\ep(s)\,d\sigma dsdt \\
  \leq& (1+\gamma) C_1\left(\gamma,v,\rho_{\mathrm{min}}^{-1}\right) \mathrm{TV}(\rho_0) \cdot T\int_0^\infty s w_\ep(s) \,ds \\
  =& (1+\gamma) C_1\left(\gamma,v,\rho_{\mathrm{min}}^{-1}\right) \mathrm{TV}(\rho_0) \cdot T\ep,
\end{align*}
\begin{align*}
  J_2 =& \int_0^T \int_{t/\gamma}^\infty \int_0^s \left( \int_{\mathbb{R}} |\partial_x\rho_0(x+\sigma)| \,dx \right) w_\ep(s) \,d\sigma dsdt \\
  \leq& \mathrm{TV}(\rho_0) \cdot T\int_0^\infty s w_\ep(s) \,ds \\
  =& \mathrm{TV}(\rho_0) \cdot T \ep,
\end{align*}
and
\begin{align*}
  J_3 =& \int_0^T \left( \int_0^t \left( \int_{\mathbb{R}} |\partial_t\rho^\ep(\tau,x)| \,dx \right) \,d\tau \int_{t/\gamma}^\infty w_\ep(s) \,ds \right) dt \\
  \leq& C_1\left(\gamma,v,\rho_{\mathrm{min}}^{-1}\right) \mathrm{TV}(\rho_0) \int_0^T t e^{-\frac{t}{\gamma\ep}} \,dt\\
  \leq& C_1\left(\gamma,v,\rho_{\mathrm{min}}^{-1}\right) \mathrm{TV}(\rho_0) \cdot \gamma T \ep.
\end{align*}
Combining the above inequalities we get the desired estimate \eqref{eq:L1estimate_q_minus_rho}.

Therefore by \eqref{eq:L1estimate_q_minus_rho}
and the convergence of $\rho^{\ep_n}\to\rho^\star$, we get $q^{\ep_n}\to\rho^\star$ in
 $\mathbf{L}^1_{\mathrm{loc}}([0,\infty) \times \mathbb{R})$ as $\ep_n\to0$.
By passing through the limit in \eqref{eq:WeakSolnDef}, we deduce that $\rho^\star$ is a weak solution of \eqref{eq:local_model}.
\end{proof}

\section{Entropy admissibility of the limit solution}\label{sec:entropy_admissbility}

In this section, we show that the weak solution to the local model \eqref{eq:local_model} obtained from the limit as $\varepsilon \to 0$ of a sequence of weak solutions to \eqref{eq:model_nonlocal_1}-\eqref{eq:model_nonlocal_3} is in fact the entropy admissible solution. This completes the theory of nonlocal-to-local limit from \eqref{eq:model_nonlocal_1}-\eqref{eq:model_nonlocal_3} to \eqref{eq:local_model} in the case of exponential kernels.

\begin{proof}[Proof of Theorem \ref{thm:limit_sol_entropy}]
Following a similar approach as in \cite{bressan2020entropy}, it suffices to establish the entropy inequality for one convex entropy, see also~\cite{dafermos2005hyperbolic}.
For this purpose, we introduce the following entropy-entropy flux pair:
\begin{align}\label{eq:entropy_pair}
  \eta(\rho) = \int_0^\rho r (1 + \gamma v(r)) \,dr, \qquad
  \psi(\rho) = \int_0^\rho r (1 + \gamma v(r)) (v(r) + r v'(r)) \,dr .
\end{align}
It is straightforward to verify that $\psi'(\rho) = \eta'(\rho) (\rho v(\rho))'$, and that $\eta(\rho)$ is strictly convex.
We claim the following entropy inequality for the nonlocal solution $\rho^\ep$ of \eqref{eq:model_nonlocal_1}-\eqref{eq:model_nonlocal_3}:
\begin{align}\label{eq:ei}
 & \int_0^\infty \int_{\mathbb{R}} \eta(\rho^\ep(t,x)) \partial_t\varphi(t,x) + \psi(\rho^\ep(t,x)) \partial_x\varphi(t,x) \,dxdt \notag\\& \geq - C\left(\gamma,v,\rho_{\mathrm{min}}^{-1},\mathrm{TV}(\rho_0),\varphi \right) \ep,
\end{align}
for all nonnegative test functions $\varphi\in\mathbf{C}^1_{\mathrm{c}} ((0,\infty)\times \mathbb{R})$,
where the constant $C=C\left(\gamma,v,\rho_{\mathrm{min}}^{-1},\mathrm{TV}(\rho_0),\varphi \right)$ is independent of $\ep$.
Assuming this claim, any limit solution $\rho^\ast$ obtained following Theorem~\ref{thm:limit_sol} satisfies the entropy inequality
\begin{align}\label{eq:entropy_inequality}
  \int_0^\infty \int_{\mathbb{R}} \eta(\rho^\ast(t,x)) \partial_t\varphi(t,x) + \psi(\rho^\ast(t,x)) \partial_x\varphi(t,x) \,dxdt \geq 0
\end{align}
for all nonnegative test functions $\varphi\in\mathbf{C}^1_{\mathrm{c}} ((0,\infty)\times \mathbb{R})$,
and thus $\rho^\ast$ is the unique entropy admissible solution of \eqref{eq:local_model}.

Now we prove the inequality \eqref{eq:ei}. Let us first assume that  $\rho_0$ is Lipschitz continuous and show \eqref{eq:ei} for Lipschitz solutions.
For simplicity we 
omit the superscript $\ep$ in $\rho^\ep$.
The equation \eqref{eq:model_nonlocal_1} can be rewritten as
\begin{align}\label{eq:entropy_limit_form}
  \partial_t \rho + \partial_x (\rho v(\rho)) = \partial_x (\rho (v(\rho) - v(q))).
\end{align}
For any nonnegative test function $\varphi\in\mathbf{C}^1_{\mathrm{c}} \left( (0,\infty)\times \mathbb{R} \right)$, multiplying $\rho(1+\gamma v(\rho))\varphi$ on both sides of \eqref{eq:entropy_limit_form} gives
\begin{align}\label{eq:entropy_limit_form_a}
  ( \partial_t \eta(\rho) + \partial_x \psi(\rho) ) \varphi = \rho (1+\gamma v(\rho)) \partial_x (\rho (v(\rho) - v(q))) \varphi.
\end{align}
Using again the directional derivative notation $\partial_y = \partial_x  - \gamma \partial_t$, we obtain the identity $\rho = q - \ep \partial_y q$.
Then \eqref{eq:entropy_limit_form_a} becomes
\begin{align}\label{eq:entropy_limit_form_b}
  &( \partial_t \eta(\rho) + \partial_x \psi(\rho) ) \varphi \notag \\
  =& ~ \gamma \partial_t (\rho^2 (v(\rho) - v(q))) \varphi + \frac12 \gamma \partial_x \left( \rho^2 \left( v(\rho)^2 - v(q)^2 \right) \right) \varphi + \rho \partial_y (\rho (v(\rho) - v(q))) \varphi.
\end{align}
Integrating \eqref{eq:entropy_limit_form_b}
and using integration by parts, we get
\begin{align*}
  \int_0^\infty \int_{\mathbb{R}} \eta(\rho) \partial_t\varphi + \psi(\rho) \partial_x\varphi \, dxdt = J_1 + J_2+ J_3,
\end{align*}
where
\begin{align*}
  J_1 &= \gamma \int_0^\infty \int_{\mathbb{R}} \rho^2 (v(\rho) - v(q)) \partial_t\varphi \,dxdt, \qquad\\
  J_2 &= \frac12 \gamma \int_0^\infty \int_{\mathbb{R}} \rho^2 \left( v(\rho)^2 - v(q)^2 \right) \partial_x\varphi \,dxdt,
\end{align*}
and
\begin{align*}
  J_3 &= \int_0^\infty \int_{\mathbb{R}} \rho \partial_y (\rho (v(q) - v(\rho))) \varphi \,dxdt \\
  &= \frac12 \int_0^\infty \int_{\mathbb{R}} \partial_y (\rho^2) (v(q) - v(\rho)) \varphi \,dxdt + \int_0^\infty \int_{\mathbb{R}} \rho^2 \partial_y (v(q) - v(\rho)) \varphi \,dxdt \\
  &= \frac12 \int_0^\infty \int_{\mathbb{R}} \rho^2 (v(\rho) - v(q)) \partial_y\varphi  \,dxdt + \frac12 \int_0^\infty \int_{\mathbb{R}} \rho^2 \partial_y (v(q) - v(\rho)) \varphi \,dxdt \\
  &\doteq \frac12 J_4 + \frac12 J_5.
\end{align*}
Repeatedly using the identity $\rho = q - \ep \partial_y q$ and integrating by parts, we compute
\begin{align*}
  J_5 &= \int_0^\infty \int_{\mathbb{R}} \rho^2 (v'(q) \partial_y q - v'(\rho) \partial_y \rho ) \varphi \,dxdt \\
  &= \int_0^\infty \int_{\mathbb{R}} q^2 v'(q) \partial_y q \varphi \,dxdt\\&\quad - \int_0^\infty \int_{\mathbb{R}} \rho^2 v'(\rho) \partial_y \rho \varphi \,dxdt - \ep \int_0^\infty \int_{\mathbb{R}} (\rho+q) v'(q) (\partial_y q)^2 \varphi \,dxdt \\
  &= \int_0^\infty \int_{\mathbb{R}}  (W(\rho)-W(q)) \partial_y\varphi \,dxdt - \ep \int_0^\infty \int_{\mathbb{R}} (\rho+q) v'(q) (\partial_y q)^2 \varphi \,dxdt \\
  &\doteq J_6 + J_7,
\end{align*}
with $W(\rho) \doteq \int_0^\rho r^2v'(r)\,dr$.

Now we have
\begin{align*}
  \int_0^\infty \int_{\mathbb{R}} \eta(\rho) \partial_t\varphi + \psi(\rho) \partial_x\varphi \, dxdt = J_1 + J_2+ \frac12 J_4 + \frac12 J_6 + \frac12 J_7.
\end{align*}
Since $\rho,q,\varphi\geq0$ and $v'(q)\leq0$, we have $J_7\geq0$. Moreover, it follows from \eqref{eq:L1estimate_q_minus_rho} that
\[ |J_1| + |J_2| + |J_4| + |J_6| \leq C_1\left(\gamma,v,\rho_{\mathrm{min}}^{-1},\mathrm{TV}(\rho_0) \right) C_2(\mathrm{supp}\varphi,\norm{\partial_t\varphi}_\infty,\norm{\partial_x\varphi}_\infty) \ep. \]
Then we obtain the inequality \eqref{eq:ei}.

The inequality \eqref{eq:ei} for initial data $\rho_0 \in \mathcal{X}$ follows from an approximation argument as in the proof of Theorem~\ref{thm:WeakSolnExistence}.
\end{proof}

Let us make some remarks on entropy pairs for the relaxation
system \eqref{model_relax_1}-\eqref{model_relax_2} and its equilibrium approximation \eqref{eq:local_model}. In the proof of Theorem~\ref{thm:limit_sol} we base the analysis directly on
the nonlocal model \eqref{eq:model_nonlocal_1}-\eqref{eq:model_nonlocal_3},
and do not rely on the rigorous justification of the entropy inequality for the relaxation
system \eqref{model_relax_1}-\eqref{model_relax_2}.
However, we remark that some  intuitive analysis based on the relaxation system \eqref{model_relax_1}-\eqref{model_relax_2} offers insight to our choice of the
entropy pair \eqref{eq:entropy_pair}.

Following the paradigm described in \cite{Chen1994}, if $(\eta,\psi)$ is any entropy-entropy flux pair for the limiting conservation law \eqref{eq:local_model}, one can construct an entropy-entropy flux pair $(H,\Psi)$ for the relaxation system \eqref{model_relax_1}-\eqref{model_relax_2} such that
\begin{align*}
  \int_0^\infty \int_{\mathbb{R}} H(\rho,q) \partial_t\varphi + \Psi(\rho,q) \partial_x\varphi + (\gamma \ep)^{-1} \partial_{q} H(\rho,q) (\rho-q) \varphi \, dxdt \geq 0,
\end{align*}
for any test function $\varphi \geq 0$, and when $\rho=q$ one has
\begin{align*}
  H(\rho,\rho)=\eta(\rho), \quad \Psi(\rho,\rho)=\psi(\rho), \quad \partial_qH(\rho,\rho)=0\,.
\end{align*}
Therefore, it holds
\begin{align*}
  &\int_0^\infty \int_{\mathbb{R}} \eta(\rho) \partial_t\varphi + \psi(\rho) \partial_x\varphi \, dxdt \\\geq & \int_0^\infty \int_{\mathbb{R}} [H(\rho,\rho)-H(\rho,q)] \partial_t\varphi + [\Psi(\rho,\rho)-\Psi(\rho,q)] \partial_x\varphi \, dxdt \\
  & -(\gamma \ep)^{-1} \int_0^\infty \int_{\mathbb{R}} [\partial_{q} H(\rho,q) - \partial_{q} H(\rho,\rho)] (\rho-q) \varphi \, dxdt.
\end{align*}
Assuming $H$ and $\Psi$ are $\mathbf{C}^2$ smooth, the right hand side is $O(\ep)$ when $\rho-q \approx \ep$.

Provided any convex $\eta$, one can construct $H$ by solving the following hyperbolic Cauchy problem \cite{Chen1994}:
\begin{align*}
  &\rho v'(q)\partial_{\rho\rho}H - (v(q)+\gamma^{-1})\partial_{\rho q}H=0,\\
  &H(\rho,\rho)=\eta(\rho),\quad \partial_qH(\rho,\rho)=0.
\end{align*}
We note that, with the simple choice of convex entropy
$\eta(\rho)=\frac12\rho^2$,
the analytic solution $H$ may be complicated.
Instead, if we choose a different convex entropy function:
\[ \eta(\rho) = \int_0^\rho r (1 + \gamma v(r)) \,dr
\]
we obtain a simple solution for $H$ as
\[H(\rho,q)=\eta(\rho)+\frac\gamma2\rho^2[v(q)-v(\rho)].\]
This motivates our choice of the entropy-entropy flux pair in
\eqref{eq:entropy_pair}.

\section{Concluding remarks}\label{sec:final}

In this paper we propose a space-time nonlocal conservation law modelling traffic flow.
The proposed model \eqref{eq:model_nonlocal_1}-\eqref{eq:model_nonlocal_2} extends the classical LWR model by introducing nonlocal velocities in the flux function.
To fit realistic traffic scenarios, the model considers time delays in the long-range inter-vehicle communication, and the model parameter $\gamma$ corresponds to the temporal nonlocal effects. In the limit as $\gamma\to 0$, our analysis shows that the model recovers a model involving only spatial nonlocality, which has been  extensively studied in the literature.

We provide well-posedness theories of the proposed model \eqref{eq:model_nonlocal_1}-\eqref{eq:model_nonlocal_2} under suitable assumptions on model parameters and the past-time condition.
Furthermore, in the special case of exponential weight kernels,
we prove convergence from solutions of the nonlocal model to the unique entropy admissible solution of the local limit equation, i.e.\ the LWR model.
The results established in this paper provide a rigorous underpinning in potential implementation of the space-time nonlocal model for the modelling of nonlocal traffic flows.

Let us make some concluding remarks on possible  generalizations of the model.
An alternative model to \eqref{eq:model_nonlocal_1}-\eqref{eq:model_nonlocal_2}
is to instead take a weighted average of vehicle velocity.
To be precise,
\begin{align*}
  & \partial_t\rho(t,x) + \partial_x(\rho(t,x) V(t,x))= 0, \\
\mbox{where}\qquad  & V(t,x) = \int_0^\infty v(\rho(t-\gamma s, x+s)) w(s)\,ds.
\end{align*}
For this model, we expect that the well-posedness and nonlocal-to-local limit
can be established in a similar fashion.
Furthermore, in future works we hope to consider more general cases where the traveling speed of nonlocal traffic information depends on additional quantities in the model.

We would also like to conduct more mathematical analysis. In this paper we show convergence of solutions of the space-time nonlocal model to the entropy admissible solution of the local model in the case of exponential weight kernels. The convergence result may be established on the nonlocal quantity $q$ for more general initial data and kernels. Such a result has been established for the nonlocal-in-space model \eqref{eq:model_nonlocal_space_1}-\eqref{eq:model_nonlocal_space_2} in \cite{colombo2022nonlocal}. We hope to show more nonlocal-to-local convergence results for the space-time nonlocal model along that direction.
Furthermore, understanding the behavior -- such as the existence, uniqueness and stability -- of traveling wave solutions of the space-time nonlocal model will shed light on the long time behavior and stability of shock waves.
In the case of exponential kernels, this is equivalent to the study of traveling waves for the relaxation system,
which could be easier to analyze.
For general kernels, an integro-differential equation is satisfied by the traveling wave profiles.
In all cases, we expect that traveling waves are local attractors for solutions.

\bibliographystyle{siam}
\bibliography{ref}

\end{document}